\documentclass[10pt, a4paper]{amsart}
\usepackage{lmodern}
\usepackage[utf8]{inputenc}
\usepackage{amsmath}
\usepackage{graphicx}
\usepackage{amssymb}
\usepackage{esint}
\usepackage[dvipsnames]{xcolor}
\usepackage{tikz}
\usepackage{xxcolor}
\usepackage{floatrow}
\usepackage{mathrsfs}
\usepackage{subcaption}
\usepackage{color}
\usepackage{amsthm}
\usepackage{epsfig}
\usepackage{pifont}
\usepackage[english]{babel}
\usepackage{hyperref}
\usepackage[normalem]{ulem}
\usepackage{tikz}
\usepackage{pgfplots}
\usepackage{caption}
\usepackage{bbm}
  
\usepackage{stmaryrd}
\usepackage{float}
\hypersetup{
	colorlinks,
	linkcolor={red!80!black},
	citecolor={blue!50!black},
	urlcolor={blue!80!black}
}
\usepackage{diagbox}
\usepackage[a4paper,top=1.5 cm,bottom=3 cm,left=1.5 cm,right=1.5 cm]{geometry}

\theoremstyle{plain}
\newtheorem{lemma}{Lemma}[section]
\newtheorem{proposition}[lemma]{Proposition}
\newtheorem{theorem}[lemma]{Theorem}

\theoremstyle{definition}

\theoremstyle{remark}
\newtheorem{remark}[lemma]{Remark}

\makeatletter
\renewenvironment{proof}[1][\proofname]{\par
  \pushQED{\qed}%
  \normalfont \topsep6\p@\@plus6\p@\relax
  \trivlist
  \item[\hskip\labelsep
        \itshape
    #1\ \ref{#1}.]\ignorespaces
}{%
  \popQED\endtrivlist\@endpefalse
}
\makeatother
\numberwithin{equation}{section}
\def\to{\rightarrow}

\usepackage{dsfont}

\newcommand{\F}{\mathbb{F}}

\newcommand{\I}{\mathbb{I}}

\usepackage{mathtools}

\allowdisplaybreaks
\usepackage[colorinlistoftodos]{todonotes}
\usepackage{url}
\usepackage{graphicx,tikz} 
\author{Marius Bargo}
 \author{Yacouba Simporé}
\title[Mosquito dynamics]{Global stabilization and emergence tracking via aquatic control in an age-structured mosquito model}
\date{}
\address[Marius BARGO]{Laboratoire LANIBIO, Université Joseph Ki-Zerbo (UJKZ), 01 BP 7021, Ouaga 01, Ouagadougou, Burkina Faso}
\email{bargomarius@gmail.com}

\address[Yacouba SIMPORE]{Chair for Dynamics, Control, Machine Learning and Numerics, Alexander Von Humboldt- Professorship, Department of Mathematics, Friedrich-Alexander-Universit\"{a}t at Erlangen-N\"{u}rnberg, Cauerstrasse 11, 91058 Erlangen, Germany, Université Yembila Abdoulaye TOGUYENI, Burkina Faso, Laboratoire LaST,  Laboratoire LANIBIO (UJKZ)}
\email{yacouba.simpore@fau.de}
\usepackage{thmbox}
\usepackage{dsfont}
\usepackage{float}
\usepackage{appendix}
\pgfplotsset{compat=1.18}
\usepackage{tikz}
\usetikzlibrary{arrows.meta}
\usetikzlibrary{arrows.meta, positioning} 
\usetikzlibrary{positioning} 
\begin{document}
\maketitle
{\bf Abstract.}  This paper presents an age-structured, non-autonomous logistic model describing the aquatic and adult stages of the dynamics of malaria-vector mosquitoes. We propose a biological control strategy targeting the aquatic compartment and implement a tracking control for its emergence. A feedback control law guarantees stabilization of the emergent population density, specifically the global asymptotic stability of the logistic model. Additionally, a feedforward controller combined with feedback is introduced to steer the emergent density toward a time-varying reference trajectory. The analytical findings are corroborated and illustrated by numerical simulations.
\section{\bf Introduction}
Vector-borne diseases transmitted by mosquitoes represent a major public health challenge, with approximately 247 million malaria cases reported worldwide according to the 2022 WHO report. Given the limitations of chemical insecticides and the emergence of resistance, genetic control methods, particularly the Sterile Insect Technique (SIT) developed by E. Knipling and the Incompatible Insect Technique (IIT) exploiting Wolbachia, have generated considerable interest in recent literature \cite{ref81, ref38, ref42, ref44, ref43, ref41, ref73}.

Although SIT and IIT techniques are extensively studied, the global regulation of vector population dynamics remains a challenge to be addressed. Most existing works model release processes using dynamical systems that are sometimes not age-structured, either in continuous formulations or periodic impulsive formulations.
These works focus essentially on elimination dynamics and optimization of release strategies, most often in non age-structured models (even though some studies consider age-structured cases, such as \cite{ref85, ref86, ref87}), and do not systematically address dynamic control of the aquatic population as a lever for global regulation.

However, certain phenomena such as the persistence of dormant eggs and latent larvae constituting a dynamic reservoir  (see \cite{ref80}) show that controlling the adult stock alone is not always sufficient: the aquatic dynamics can govern population recovery and determine the success of interventions.

Our mathematical approach falls within the framework of age-structured control systems, clearly distinguishing the aquatic phase (eggs, larvae, pupae) and the adult phase. We model practical measures (reduction of breeding sites, targeted treatment of larval habitats, drainage, larvicide treatments, biological releases, or introduction of predators such as affinis fish) as multiplicative controls applied to the aquatic dynamics. The adopted control scheme combines a feedforward component to impose an emergence trajectory and a feedback component to stabilize the dynamics in the presence of perturbations.

The stabilization of population dynamics models by feedback or feedforward control is not new \cite{ref82, ref63, ref84, ref70, ref69, ref78, ref77}. However, the application of these systematic control techniques to the age-structured multiphase dynamics of mosquito populations remains largely unexplored, despite its potential to effectively regulate the overall vector density by acting on the aquatic stage.

This work proposes a rigorous mathematical framework: it establishes a global stability result and develops a control strategy for tracking a dynamic reference trajectory. The approach applies to a non-autonomous age-structured logistic model (aquatic and adult stages) subject to a bounded multiplicative control targeting the aquatic population. The mathematical analysis shows that control of the aquatic dynamics is sufficient to regulate the overall species dynamics, thus constituting a complementary strategy to existing genetic vector control measures.

The remainder of the paper is organized as follows. First, we present the model and its assumptions in \textbf{Section \ref{sec2}}. Second, in \textbf{Section \ref{sec4}}, we establish the preliminary tools required for proving the main results. Next, \textbf{Section \ref{sec5}} presents these main results. Then, \textbf{Section \ref{sec6}} is devoted to the study of the well-posedness of the model. Subsequently, in \textbf{Section \ref{sec7}}, we provide the proofs and methodological arguments supporting the stated results. Moreover, \textbf{Section \ref{sec8}} contains numerical simulations that validate the theoretical findings. Finally, \textbf{Section \ref{sec9}} presents the conclusion and discusses possible directions for future research.

\section{\bf Problem setting}\label{sec2}
We extend the formulations of mosquito dynamics models from \cite{ref38, ref42, ref44, ref43, ref41} by explicitly incorporating age structure and a logistic term representing demographic pressure linked to resource availability. Building on this framework, we develop an age-structured dynamic model that distinguishes three population compartments :
\begin{itemize}
\item $I$, the aquatic (immature) stage;
\item $F$, adult females;
\item $M$, adult males.
\end{itemize}
We then propose a vector-control strategy based on the introduction of a predator targeting the aquatic compartment (biological control). Biological control exploits the deliberate introduction of natural enemies to suppress pest populations, especially when such pests expand unchecked in the absence of their usual predators (the ecological‐release paradigm). One classic example is the use of the mosquitofish  Gambusia affinis, introduced into Algeria in 1928 (and earlier in Europe, circa 1921) to prey upon anopheline larvae and curb malaria transmission. Native to Central America and Florida,  G. affinis thrives in diverse freshwater habitats and remains one of the most effective biological control agents against mosquitoes, readily integrating with existing vector‐management strategies.
In the age-structured logistic model considered here, $P(t)$ denotes a time-dependent control representing human interventions (such as drainage, the introduction of larvivorous fish, and environmental management) that act solely on the aquatic mosquito cohort.
Upon emergence, adults are allocated to females and males according to a fixed sex ratio $r\in(0,1)$.  To account for the emergence of the aquatic population, 
we introduce the function $w(a),$ which represents the age‑dependent emergence rate. Consider the following model

\begin{equation}\label{eq3.1}
\left\lbrace 
\begin{array}{ll}
\partial_tI(a,t)+\partial_aI(a,t)+\mu(a,p(t))I(a,t)=\Gamma(t) I(a,t)\left(1-\dfrac{\gamma(t)}{K(t)}\displaystyle\int_{0}^AI(a,t)da\right)-I(a,t)P(t) &\text{ in }Q=(0,A)\times\mathbb{R}_+,  
\\\partial_tF(a,t)+\partial_aF(a,t)+\mu_{F}(a)F(a,t)=-\gamma(t)F(a,t)\displaystyle\int_{0}^A  F(a,t)da &\text{ in }Q=(0,A)\times\mathbb{R}_+,
\\\partial_tM(a,t)+\partial_aM(a,t)+\mu_M(a)M(a,t)=-\gamma(t) M(a,t)\displaystyle\int_{0}^A M(a,t)da &\text{ in }Q=(0,A)\times\mathbb{R}_+,
\\I(0,t)=\displaystyle\int_{0}^A\beta(a,m(t))F(a,t)da,&\text{ in }Q_+=\mathbb{R}_+,
\\F(0,t)=r\displaystyle\int_0^Aw(a)I(a,t)da,\; M(0,t)=(1-r)\displaystyle\int_0^Aw(a)I(a,t)da,&\text{ in }Q_+=\mathbb{R}_+,
\\m(t)=\int_{0}^A\lambda(a)M(a,t)da,\\
\\I(a,0)\geq  0,\; F(a,0)\geq  0,\; M(a,0)\geq  0, &\text{ in }Q_A=(0,A),
\end{array}
\right.
\end{equation} 
where $Y_0=\left(I(a,0), F(a,0),  M(a,0)\right)^T,$ denotes the initial state in $\mathcal H_2^3 = \bigl(L^2(0,A)\bigr)^3.$ System $\eqref{eq3.1}$ describes the dynamics of an Anopheles mosquito population by distinguishing three age‑structured cohorts $a$: the aquatic population $I(a,t)$, adult females $F(a,t)$, and adult males $M(a,t)$.
 In the aquatic phase, each cohort experiences natural mortality $\mu\bigl(a,p(t)\bigr)\,I(a,t),$
 where $p(t)=\displaystyle\int_{0}^{A}I(a,t)\,\mathrm{d}a$ represents the concentration of the aquatic population subject to predation or other stressors. System \eqref{eq3.1} with control $P(t)$ can thus be studied for global asymptotic stability, while a distributed control acting across all ages of the aquatic cohort, akin to the framework in \cite{ref61}, raises natural questions of controllability under predation pressure. In our framework, the time‐dependent function $P(t)$ is interpreted not as the intrinsic dynamics of a Gambusia affinis population, but as a unified control parameter representing all human‐driven interventions against the aquatic mosquito stage, whether by fish releases, habitat drainage, larviciding, or other larval‐reduction measures.  In practice, these activities are planned and scheduled by antimalarial programs according to predetermined frequencies, dosages, and target areas; accordingly, $P(t)$ appears in the model as an exogenous mortality rate term, $-I(a,t)\,P(t)$, applied uniformly across the aquatic cohort.  This aggregation of disparate control actions into a single, time‐varying parameter greatly simplifies the system by obviating the need for an extra differential equation for the predator, while still capturing the combined ecological and operational constraints of vector‐control campaigns.
Furthermore, although one could introduce interspecific terms in the adult mortality rates to reflect resource competition (e.g. for nectar or resting sites), we assume that adult female and male death rates depend solely on age.  This assumption aligns with the natural separation of feeding niches (blood meals for females versus nectar for males) and allows us to concentrate the mathematical analysis on the stability effects of the aquatic‐stage control $P(t)$.  
Since model \eqref{eq3.1} is a non-autonomous logistic model, to carry out its qualitative analysis we may likewise replace the time-dependent functions $K(t)$, $\Gamma(t)$, and $\gamma(t)$  by their respective mean values $K^*$, $\Gamma^*$, and $\gamma^*$. In the dynamic case, the functions $K(t)$, $\Gamma(t)$, and $\gamma(t)$ are assumed to be continuous and bounded on the interval $\mathbb{R}_+,$ namely assuming hypothesis
\begin{align*}
\textbf{(H1)} \begin{cases}
K(t),\; \Gamma(t),\; \gamma(t) \in L^{\infty}(Q_+),
\\\Gamma(t)\ge 0,\gamma(t)\ge 0,\;\text{a.e.}\; t\in \mathbb{R}_+,
\\\text{There exists}\; \epsilon>0\;\text{such that}\;0<\epsilon \le K(t),\;\text{a.e.}\; t\in \mathbb{R}_+.
\end{cases}
\end{align*}
We can assume that the functions $K(t),\, \gamma(t),\,\Gamma(t)$ are periodic, as in \cite{ref83}, to reflect seasonal environmental variability. When parameters vary periodically, the model explicitly accounts for seasonal drivers such as temperature and precipitation. In the stochastic framework, the incorporation of random variability further captures environmental uncertainty.
Throughout the remainder of the paper we work in a Hilbert space; more precisely, the initial conditions $I(a,0),\;M(a,0),\;F(a,0)$ belong to $L^2(0,A).$
Moreover, we adopt the following standing hypotheses (unless otherwise stated):
\begin{align*}
\textbf{(H2)}\begin{cases}
\mu_i(a),\; \mu(a,p)\geq 0 \quad \text{a.e. on }(0,A), \\[0.3em]
\mu_i \in L^1_{\rm loc}(0,A), \quad \displaystyle\int_0^{A}\mu_i(a)\,da = +\infty,i\in\lbrace F,M\rbrace,\\
\mu(a,p) \in L^1_{\rm loc}(0,A), \quad \displaystyle\int_0^{A}\mu(a,p)\,da = +\infty.
\end{cases}
\qquad
\textbf{(H3)}\begin{cases}
\beta(.,m),\; w(.) \in L^\infty(0,A),  \\[0.3em]
\beta(a,m),\; w(a)\ge 0 \ \text{a.e. on } (0,A).
\end{cases}
\end{align*}
The description of the parameters is given in Table \ref{t1} below.
\begin{table}[ht]
\centering
\begin{tabular}{c|p{10cm}}
\hline
Parameter & Description \\
\hline
$\lambda(.)$ & Fertility function of male individuals. \\
$r \in (0,1)$ & Primary sex ratio in offspring. \\
$\beta(.,m)$ & Mean number of eggs that a single female can deposit on average per day. \\
$w(.)$ & the age‑dependent emergence rate. \\
$\mu(.,p(t)),\,\mu_{F}(.),\,\mu_M(.)$ & Mean death rates of immature individuals (density‐dependent and independent),  females and  males, respectively. \\
$\gamma(t)$ & Competition parameter. \\
$K(t)$  & Carrying capacity related to the amount of available nutrients and space. \\
$\Gamma(t)$  & Growth rate. \\
\hline
\end{tabular}
\caption{Description of the parameters}
\label{t1}
\end{table}

\section{\bf Preliminary}\label{sec4}
This section is devoted to presenting the tools required for the stability analysis. The steady state of our model \eqref{eq3.1} is of particular importance, especially for our further research on stability analysis. The steady state of \eqref{eq3.1}, denoted by $(I^*\ge 0,\;F^*\geq 0,\; M^*\geq 0)$, must be a solution of 

\begin{equation}\label{eq3.11}
\left\lbrace 
\begin{array}{ll}
\partial_a\,I^*(a)\;+\;\bigl(\mu\bigl(a,p^*\bigr)\;+\;\zeta_I\bigr)\,I^*(a)\;=\;0, 
& \text{ in}\;Q_A,\\[1em]
\partial_a\,F^*(a)\;+\;\bigl(\mu_F(a)\;+\;\zeta_F\bigr)\,F^*(a)\;=\;0, 
& \text{ in}\;Q_A,\\[1em]
\partial_a\,M^*(a)\;+\;\bigl(\mu_M(a)\;+\;\zeta_M\bigr)\,M^*(a)\;=\;0, 
& \text{ in}\;Q_A,\\[1em]
I^*(0)=\displaystyle \int_{0}^{A} \beta\bigl(a,m^*\bigr)F^*(a)da,\\[1em]
F^*(0)=r\displaystyle\int_0^Aw(a)I^*(a)da,\\[1em]
M^*(0)=(1-r)\,\displaystyle\int_0^Aw(a)I^*(a)da.
\end{array}
\right.
\end{equation}
where

\begin{align}\label{e3.47}
    \zeta_I 
\;=\; \frac{\Gamma^*\gamma^*}{K^*}\,\displaystyle \int_{0}^{A} \,I^*(a)\,da \;+\; P^*-\Gamma^*,
\quad
\zeta_F 
\;=\gamma^*\;\displaystyle \int_{0}^{A} \,F^*(a)\,da,
\quad
\zeta_M 
\;=\gamma^*\; \displaystyle \int_{0}^{A} \,M^*(a)\,da.
\end{align}

The corresponding solutions are given by

\begin{align}\label{e3.48}
    I^*(a)= I^*(0)\,\underbrace{e^{-\displaystyle\int_{0}^{a}\bigl[\mu_I(s,p^*) + \zeta_I\bigr]\,ds}}_{\Tilde{I}^*(a)},\;F^*(a) 
= F^*(0)\,\underbrace{e^{-\displaystyle\int_{0}^{a}\bigl[\mu_F(s) + \zeta_F\bigr]\,ds}}_{\Tilde{F}^*(a)},\;M^*(a) 
&= M^*(0)\,\underbrace{e^{-\displaystyle\int_{0}^{a}\bigl[\mu_M(s) + \zeta_M\bigr]\,ds}}_{\Tilde{M}^*(a)},
\end{align}
where $\zeta_I$ and $\zeta_F$ are solutions of 
\begin{align}\label{e3.49}
r\displaystyle\int_0^Aw(a)\tilde{I}^*(a)da\displaystyle\int_{0}^A\beta(a,m)\tilde{F}^*(a)da=1.
\end{align}
We rewrite these solutions of the form
\begin{align}\label{e3.50}
    F^*(a)=rI^*(0)\displaystyle\int_0^Aw(a)\tilde{I}^*(a)da\tilde{F}^*(a),\;M^*(a)=(1-r)I^*(0)\displaystyle\int_0^Aw(a)\tilde{I}^*(a)da\tilde{M}^*(a).
\end{align}
\begin{remark}\label{remark3.3}
 Ensuring the stability of $I$ automatically ensures the stability of both $M$ and $F$. Indeed, we have from \eqref{eq3.11} 
\begin{align}\label{e3.51}
  I^*(0)=\dfrac{K^*}{\Gamma^*\gamma^*}\dfrac{(\left[\zeta_I+\Gamma^*\right]-P^*)}{\displaystyle\int_0^A\tilde{I}^*(a)da}>0,\quad\; P^*\in (0,\zeta_I+\Gamma^*).
\end{align}
It is noteworthy that, according to this expression, increasing the equilibrium control $P^*$ leads to a pronounced reduction in the steady-state abundance of both male and female mosquitoes. In other words, bolstering the predator population has a directly dampening effect on mosquito dynamics, underscoring the decisive influence of controlling the aquatic phase on the system’s overall behavior. Thus, our analysis of the stability of system \eqref{eq3.1} reduces to the stabilization of the aquatic population.
\end{remark}
Remark \ref{remark3.3} is  consistent with the fundamental structure of model \eqref{eq3.1}, since the adult population arises from the emergence of the aquatic compartment through the defined emergence rate $w$. Therefore, stabilizing the aquatic population is equivalent to stabilizing the dynamics of the entire vector population.
\begin{lemma}\label{le3.4}
Consider the following transformation
\begin{align}\label{e3.52}
 \left[\begin{array}{c}
\eta_I(t) \\ 
\psi_I(t-a)\\
\psi_F(t-a)\\
\psi_M(t-a)
\end{array}  \right]=\left[\begin{array}{c}
\ln[\Pi_I(I(t))] \\ 
\dfrac{I(a,t)}{I^*(a)\Pi_I(I(t))}-1\\
\dfrac{F(a,t)}{F^*(a)\Pi_I(I(t))}-1\\
\dfrac{M(a,t)}{M^*(a)\Pi_I(I(t))}-1
\end{array} \right],
\end{align}
where
\begin{align}\label{e3.53}
\Pi_I(I(t))=\dfrac{\langle \pi_{0,I}, I(t)\rangle_{L^2(0,A)}}{\langle\pi_{0,I}, I^*\rangle_{L^2(0,A)}},
\end{align}

with  $\pi_{0,I},$ the continuous function of the form
\begin{align}\label{e3.54}
   \pi_{0,I}(a)= \displaystyle\int_a^{A}\beta(s,m)e^{\int_s^a(\zeta_I+\mu_I(l,p)dl}ds.
\end{align}
Moreover, the variables $\psi_i$ and $\eta_I$ satisfy:

\begin{align}\label{e3.55}
\begin{cases}
    \partial_t\eta_I(t)=\zeta_I-P(t)+\Gamma(t)-\frac{\Gamma(t)\gamma(t)}{K(t)}e^{\eta_I}\displaystyle\int_{0}^A (1+\psi_I(t-a))I^*(a)da,\\
\eta_{I}(0)
= \ln\!\bigl(\Pi[I_{0}]\bigr)
= \eta_{I,0},
\end{cases}
\end{align}


\begin{align}\label{e3.56}
\begin{cases}
\psi_I(t)=\displaystyle\int_0^Ag_F(a)\psi_F(t-a)F^*(a)da,\\
\psi_F(t)=\displaystyle\int_0^Ag_I(a)\psi_I(t-a)I^*(a)da,\\
\psi_M(t)=\displaystyle\int_0^Ag_I(a)\psi_I(t-a)I^*(a)da,\\
 \psi_{i}(-a)
=\frac{i_0(a)}{i^{*}(a)\,\Pi[i_0]}\;-\;1
= \psi_{i,0}(a).
\end{cases}
\end{align}

with 
\begin{align}\label{e3.57}
g_F(a)
=\frac{\beta(a,m)F^{*}(a)}
{\displaystyle\int_{0}^{A}\beta(a,m)F^{*}(a)\,\mathrm{d}a},\; g_I(a)
=\frac{w(a)I^{*}(a)}
{\displaystyle\int_{0}^{A}w(a)I^{*}(a)\,\mathrm{d}a},
\;\text{and}
\;\int_{0}^{A}g_F(a)\,\mathrm{d}a=1,\;\int_{0}^{A}g_I(a)\,\mathrm{d}a=1.
\end{align}
The unique solutions are then given by:

\begin{align}\label{e3.58}
I(a,t)
={I}^{*}(a)\,\bigl(1+\psi_{I}(t-a)\bigr)\,e^{\eta_{I}(t)},\; F(a,t)
={F}^{*}(a)\,\bigl(1+\psi_{F}(t-a)\bigr)\,e^{\eta_{I}(t)},\; M(a,t)
={M}^{*}(a)\,\bigl(1+\psi_{M}(t-a)\bigr)\,e^{\eta_{I}(t)}.
\end{align}
\end{lemma}
\begin{proof}
For the proof of this Lemma, see  Appendix~\ref{annexe:A}.
\end{proof}
\begin{remark}
The densities $(I, F, M)$ adopt here the form given in \eqref{e3.58}, which distinguishes them from the expressions used in \cite{ref45}. The system \eqref{eq3.1} couples an aquatic dynamic with an adult dynamic emerging from the former : the aquatic population ensures the renewal of the adult population. The control $P$ acts exclusively on the aquatic population, so that any modification of its dynamics has a significant impact on the adult population (see \eqref{e3.50}-\eqref{e3.51}). Consequently, controlling the aquatic population amounts to controlling the adult population, and, more generally, the overall dynamics of the system. The diagram below illustrates the fundamental strategy used to establish the global stability of model \eqref{eq3.1}.
\begin{center}
\resizebox{\textwidth}{!}{
\begin{tikzpicture}[node distance=2.5cm, >=Stealth,
  block/.style={draw, rectangle, rounded corners, align=center, minimum width=3cm, minimum height=1cm}]
  
  \node[block] (dynamics) {Resulting dynamics:\\
    \(I(a,t),\,F(a,t),\,M(a,t)\)};
    
  \node[block, right=of dynamics] (composition) {Compose:$j\in \lbrace I, F \rbrace,\; X^*\in \lbrace  I^*,F^*,M^*\rbrace$\\
    \(\bigl(1+\psi_j(t-a)\bigr)\,X^*(a)\,e^{\eta_I(t)}\)};

     \node[block, right=of composition] (condition) {stability of $\eta_I,\;\psi_j$:\\
    \(\eta_I\to 0,\;\psi_j\to 0\)};
    
  \node[block, right=of condition] (equilibrium) {Equilibrium :\\
    \(I^*(a),\;F^*(a),\;M^*(a)\)};

  \draw[->] (dynamics) -- (composition);
  \draw[->] (composition) -- (condition);
  \draw[->] (condition) -- (equilibrium);

\end{tikzpicture}
}
\end{center}
\end{remark}
To establish the main results, we assume the following {\bf Assumptions} {\bf (H4)} and {\bf (H5)} :
 \begin{enumerate}
\item[]  {\bf Assumption (H4):  $\psi_I\equiv 0,$}
\item[]  {\bf Assumption (H5):  $\psi_I\neq 0.$}
 \end{enumerate}
\section{\bf Main results}\label{sec5}
In this section we briefly describe our strategy and the principal results.  We prove global stabilization of a non-autonomous logistic mosquito population model using a reduced control strategy acting exclusively on the aquatic stage. In the stability analysis of $\eta_I,$ we introduce the following Lyapunov candidate function :

 \begin{align}\label{e3.72}
     V_{I}(\eta_{I})=\;\int_{0}^{\eta_{I}}\phi_{I}(\alpha)\,\mathrm{d}\alpha
 \;
 \end{align}
 \begin{align*}
     =\;k_I\bigl(e^{\eta_{I}}-\eta_{I}-1\bigr)
 \end{align*}
 \begin{align*}
     =\phi_I(\eta_I)-k_I\eta_I.
 \end{align*}
 
 This function satisfies the Lyapunov conditions:
\begin{itemize}
 \item  $V_{I}(0)=0$,
 \item   for all $\alpha\neq0,\;V_{I}(\alpha)>0$ and $\lim_{\alpha\to\infty}V_{I}(\alpha)=+\infty$.
\end{itemize}
\begin{theorem}\label{th3.6}
Assume {\bf Assumption (H4)}. Define the control
\begin{align}\label{e3.75} 
 P(t)=P^*+(\Gamma(t)-\Gamma^*)+k_I\Big(\frac{\Gamma^*\gamma^*}{K^*}-\frac{\Gamma(t)\gamma(t)}{K(t)}\Big).
\end{align}
With the Lyapunov function $V_I$, the closed-loop system \eqref{eq3.1} is globally asymptotically stable.
\end{theorem}
 \begin{remark}\label{remark4.2}
It was proved in \cite{ref69} that the state $\psi_I$ of the internal
dynamics are restricted to the sets 
\begin{align*}
\mathcal{S}=\left\lbrace \psi_i\in\; C^0((-A,0);(-1,\infty)) : P(\psi_i)=0\wedge\psi_I(0)=\displaystyle\int_0^Ag_F(a)\psi_I(-a)da\right\rbrace,
\end{align*}
where
\begin{align*}
P(\psi_I)=\dfrac{\displaystyle\int_0^A\psi_I(-a)\displaystyle\int_a^Ag_F(s)dsda}{\displaystyle\int_a^Aag_F(a)da},
\end{align*}
and that the state $\psi_I$ is globally
exponentially stable in $\mathcal{L}^{\infty}$ norm, which means that there exist $M_i>1,\sigma_i>0$ such that 
\begin{align*}
    \Vert\psi_i(t-a) \Vert \leq M_ie^{-\sigma_i t} \Vert\psi_{i,0} \Vert_{\infty},\; i\in \lbrace I, F, M\rbrace.
\end{align*}
\end{remark}
Thanks to \eqref{e3.57}, 
we then make the following hypothesis \cite{ref69} :
\begin{enumerate}
     \item[]  {\bf Assumption (H6): }
    There exist constants $\kappa_I,\;\kappa_F>0$ such that%
\begin{align*}
\int_{0}^{A}\Bigl|\,
 g_F(a)\,
 \;-\;z_I\kappa_I\!\int_a^A g_F(s)\,\,ds\Bigr|\;da
 \;<\;1,\; \int_{0}^{A}\Bigl|\,g_I(a)\,\;-\;z_F\kappa_F\!\int_a^Ag_I(s)\,ds\Bigr|\;da\;<\;1
\end{align*}

 where $ z_I = \Bigl(\int_{0}^{A}a\,g_F(a)\,da\Bigr)^{-1},\; z_F = \Bigl(\int_{0}^{A}a\,g_I(a)\,da\Bigr)^{-1}.$  Let $\sigma>0$  be a sufficiently small constant that satisfies the inequality 
 \begin{align*}
\int_{0}^{A}\Bigl|\,g_I(a)\,\;-\;z_F\kappa_F\!\int_a^A g_I(s)\,ds\Bigr|\;e^{\sigma_I a}da<\;1,\; 
 \int_{0}^{A}\Bigl|\,
  g_F(a)\,
 \;-\;z_I\kappa_I\!\int_a^A  g_F(s)\,ds\Bigr|
 \;e^{\sigma_I a}da<\;1.
 \end{align*}
 \end{enumerate}
Before stating the main result of this section, we define the following functions. Let the functional
\begin{align}\label{e3.89}
    G_I(\psi_I)
= \dfrac{\max_{a\in(0,A)}\bigl|\psi_I(t-a)\bigr|\,e^{-a\sigma_I}}{1+\max(0,\min_{a\in(0,A)}\psi_I(t-a))},
\end{align}
whose Dini derivative satisfies (see \cite{ref69})
\begin{align}\label{e3.90}
    D^+\bigl(G_I(\psi_{I,t})\bigr)
\;\le\;
-\sigma_I\,G_I(\psi_{I,t})
\end{align}

We then define the following Lyapunov function
\begin{align}\label{e3.91}
    V(\eta_I,\psi_I)
= V_I(\eta_I)\;+\;\frac{\gamma_1}{\sigma_I}h(G_I(\psi_I)).
\end{align}
with
\begin{align}\label{e3.92}
h(p)=\displaystyle\int_0^p\frac{1}{z}(e^z-1)^2dz.
\end{align}
This function $h$ is positive definite and radially unbounded \cite[Lemma 3]{ref77}.
\begin{theorem}\label{th3.8}
Under {\bf Assumption (H6)}, system \eqref{eq3.1} is globally asymptotically stabilizable, and the control remains uniformly bounded. Moreover,  the control satisfies \(P(t)>0\) for all \(t>0\), for every initial condition \(\eta_I(0)\) belonging to the largest level set of $V(\eta_I,\psi_I)$ within the set

 \begin{align}\label{e3.98}
\mathcal{A} &= 
\left\{ (\eta,\psi)\in\mathbb{R}\times\mathcal{S}\ \middle|\ 
\begin{aligned}
&\eta \le \ln\!\left(\sqrt{\frac{\gamma_1}{Ck_I}}\right),\qquad \gamma_1 > Ck_I\\[4pt]
& P^*+(\Gamma(t)-\Gamma^*)+k_I\Big(\frac{\Gamma^*\gamma^*}{K^*}-\frac{\Gamma(t)\gamma(t)}{K(t)}\Big)>0
\end{aligned}
\right\}.
\end{align}
\end{theorem}

The main result below consists to design a control such that the density resulting from the emergence of the aquatic population follows a given reference trajectory $y_d$. To this end, we associate the following output to system \eqref{eq3.1} 
\begin{align}\label{eq8.1}
y(t)=\int_{0}^{A} w(a)I(a,t)\mathrm{d}a.
\end{align}
Moreover, we denote by $\mathcal{K}_\infty$ the class of continuous functions $\kappa:\mathbb{R}_*^{+}\to\mathbb{R}_*^+$ that are strictly increasing, satisfy $\kappa(0)=0$, and for which $\displaystyle\lim_{s\to+\infty}\kappa(s)=+\infty$.
\begin{theorem}\label{th5.4}
Consider a positive
reference trajectory 
\begin{align}
y_d(t)\in\left(0,\frac{\min(2,\alpha)}{4p^*}\min\left(1,\frac{P_{\max}-P_{\min}}{2}\right)\right]
\end{align} 
as well as the dynamic controller
\begin{align}\label{eq5.7}
    P(t)=P_{FF}(t)+P_{FB}(t)\in (P_{\min},P_{\max})
\end{align} where 
\begin{align}
P_{FF}(t)=&\zeta_I-\dfrac{\partial_t y_d(t)}{y_d(t)}+\Gamma(t)-p^*\frac{\Gamma(t)\gamma(t)} {K(t)}y_d(t)\in \left[\frac{4p^*}{\min(2,\alpha)}y_d(t)+P_{\min},P_{\max}-\frac{4p^*}{\min(2,\alpha)}y_d(t)\right]
\end{align}
with $p^*=\displaystyle\int_0^A\dfrac{I^*(a)}{y^*}da=\displaystyle\int_{0}^Ap(a)da$
\begin{align}\label{equation4.12}
P_{FB}(t)=&\alpha\displaystyle\ln\frac{y(t)}{y_d(t)}
\end{align}
with $\alpha>0$ denotes a control gain.  There exists a positive constant $L>0$ and a function $\kappa\in\mathcal{K}_{\infty}$, such that the unique solution of the closed loop \eqref{eq3.1} involving the dynamic controller \eqref{eq5.7} and arbitrary initial conditions $I_0$ satisfies the following estimate for all $t\geq 0$
\begin{align}\label{eq5.11}
\left\Vert \ln \displaystyle\frac{I(a,t)}{I_d(a,t)}\right\Vert_{\infty}\leq e^{-\frac{L}{4}t} \kappa\left(\left\Vert \ln \displaystyle\frac{I_0(a)}{I_d(a,0)}\right\Vert_{\infty}\right).
\end{align}
\end{theorem}
The goal of the tracking control for the emerging population goes beyond stabilization. It consists in ensuring that the adult population density follows a desired trajectory $y_d$ (or $y(t)=\displaystyle\int_0^Aw(a)I^*(a)da$), which we define as a trajectory along which the risk of vector-borne disease transmission is reduced. The pioneering work of Ronald Ross, which showed that reducing the mosquito population below a critical threshold prevents malaria transmission \cite{ref5}, supports the relevance of this approach. In this sense, driving the emerging population to follow the trajectory $y_d$ amounts to directing the density toward epidemiologically safe regimes.
\begin{remark}
The emergence function may vary in time (periodically or not) under the influence of rainfall, temperature, or climate in general. Within the framework of our non-autonomous logistic model \eqref{eq3.1}, these environmental covariates are incorporated into the dynamics; consequently we parameterize the emergence function as a function of age, analogous to the age-dependent fertility functions, in order to model the renewal of the adult population.
\end{remark}
\section{\bf Well-posedness}\label{sec6}
We establish the well‑posedness of the time‑evolution problem by means of the semigroup approach.  To this end, let $\mathcal H_2^3 = \bigl(L^2(0,A)\bigr)^3,$
  and define the linear operator

 $$
   \mathcal A_m : D(\mathcal A_m)\subset\mathcal H_2^3\;\longrightarrow\;\mathcal H_2^3,
   \quad
   \mathcal A_m\varphi
   = -\partial_a\varphi \;-\; D(a,p)\,\varphi,\quad\text{where}\; \varphi=(\varphi_I,\varphi_{F},\varphi_M)
 $$
 with

\begin{align*}
D(\mathcal A_m)
= \Bigl\{\varphi\in\mathcal H_2^3 :\;
\varphi\text{ is a.c. on }[0,A],\;
&\varphi_I(0)=\int_{0}^{A}\beta(a,m)\,\varphi_{F}(a)\,da,
&\varphi_{F}(0)=r\int_{0}^{A}w(a)\,\varphi_I(a)\,da,\\
&\varphi_M(0)=(1-r)\int_{0}^{A}w(a)\,\varphi_I(a)\,da,
&-\partial_a\varphi - D(a,p)\,\varphi\in\mathcal H_2^3
\Bigr\}.
\end{align*}

 In block‑diagonal notation,

 $$
   \mathcal A_m
   = \mathrm{diag}\bigl(-\partial_a-\mu_I,\;-\partial_a-\mu_{F},\;-\partial_a-\mu_M\bigr).
 $$
The operator $(\mathcal{A}_m, D(\mathcal{A}_m))$ is the infinitesimal generator of a strongly continuous semigroup $\mathcal{T}= (\mathcal{T}_t)_{t\geq 0}$ on $\mathcal{H}^3_2.$ Finally, the nonlinear fonction $f:\mathcal H_2^3\to\mathcal H_2^3$ is defined component‑wise by

\begin{align}\label{e6.1}
f\bigl(I,F,M,\bigr)
  = \bigl(I\,f_1,\,F\,f_2,\,M\,f_3\bigr)^\top
\end{align}
with 
\begin{align}\label{e3.40}
f_1=\Gamma(t)\left(1-\dfrac{\gamma(t)}{K(t)}\displaystyle\int_{0}^A F(a,t)da\right)-P(t),\; f_2=-\gamma(t)\displaystyle\int_{0}^A F(a,t)da,\;f_3=-\gamma(t)\displaystyle\int_{0}^AM(a,t)da.
\end{align}
Let 
\begin{align}\label{e3.42}
    Y(t)=\left( I(a,t), F(a,t),  M(a,t)\right)^T\in D(\mathcal{A}_m)
\end{align}
thus, we can rewrite the system \eqref{eq3.1} as an abstract Cauchy problem
\begin{equation}\label{e3.43}
\left\lbrace 
\begin{array}{ll}
\partial_tY(t)=(\mathcal{A}_m+H(Y(t)))Y(t),&\text{ in}\;Q_+\\
Y(0)=Y_0
\end{array}
\right.
\end{equation} 
where
\begin{align}
    Y_0=\left(I(a,0), F(a,0),  M(a,0)\right)^T,\; \text{and}\; f(Y(t))=H(Y(t))Y(t).
\end{align}

Investigating the well-posedness of system \eqref{eq3.1} reduces to studying equation \eqref{e3.43} along with its initial.  
\begin{proposition}
Under {\bf assumptions (H1)-(H2)-(H3)}, model \eqref{eq3.1} is well-posed.
\end{proposition}
\begin{proof}
The function $f$ defined in \eqref{e6.1} is globally Lipschitz in $\mathcal H_2^3$, which implies that the associated matrix $H$ is bounded.
Since $\mathcal{A}_m$ is the generator of a $C_0$-semigroup on $\mathcal H_2^3$, and $H$ is a bounded, it follows from the Phillips theorem that $\mathcal{A}_m+H$ also generates a $C_0$-semigroup on $\mathcal H_2^3$.
Therefore, the Cauchy problem \eqref{e3.43} admits a unique mild solution in $\mathcal H_2^3.$
\end{proof}
\begin{remark}
By applying the method of characteristics to the system \eqref{eq3.1}, one finds that, for every $(a,t)\;\in Q$, the solutions of \eqref{eq3.1} can be written as follows:

\begin{align}\label{e3.45}
\begin{cases}
I(a,t)=I\bigl(0,\,t - a\bigr)\,e^{-\!\displaystyle\int_{0}^{\,a}\mu_{I}\bigl(\alpha,\,p\bigl(\alpha - (a - t)\bigr)\bigr)\,d\alpha 
\;+\;\int_{\,t - a}^{\,t}R_{I}(s)\,ds},
&  R_{I}(s)
=
\Gamma(t)\left(1-\dfrac{\gamma(t)}{K(t)}\displaystyle\int_{0}^AI(x,s)dx\right)
\;-\;P(s),\\
 F(a,t)=  F(0,t-a)\,e^{-\!\displaystyle\int_{0}^{\,a}\mu_{F}(\alpha)\,d\alpha 
   \;+\;\int_{\,t - a}^{\,t}R_{F}(s)\,ds},
   &   R_{F}(s)=-\displaystyle\int_{0}^{A}\gamma(t)\,F(x,s)\,dx,\\
M(a,t)= M(0,t-a)\,e^{-\!\displaystyle\int_{0}^{\,a}\mu_{M}(\alpha)\,d\alpha 
   \;+\;\int_{\,t - a}^{\,t}R_{M}(s)\,ds},
   &   R_{M}(s)
   =
   -\,\int_{0}^{A}\gamma(t)\,M(x,s)\,dx.
\end{cases}
\end{align}
\end{remark}
\section{\bf Proof of the main results}\label{sec7}
 This section is reserved for the proofs of the main results.
 \subsection{Proof of Theorem \ref{th3.6} \& Theorem \ref{th3.8}}
 This step focuses on the mathematical analysis of the stability of the model \eqref{eq3.1}. The objective is to examine how the biological control $P$, when applied to the aquatic population, influences the overall dynamics of the system.
 \subsubsection{\bf Stability in the absence of a delay term (Theorem \ref{th3.6})}\label{s341}
We arrive, under {\bf Assumption (H4)}, at the following system from \eqref{e3.55}
\begin{align}\label{e3.69}
   \partial_t\eta_I(t)=\zeta_I-P(t)+\Gamma(t)-\dfrac{\Gamma(t)\gamma(t)}{K(t)}e^{\eta_I}\displaystyle\int_{0}^A I^*(a)da.
\end{align}

From equation \eqref{e3.47}, by setting 
\begin{align}\label{e3.70}
    k_I=\displaystyle\int_{0}^A I^*(a)da,\;\phi_I(\eta_I)=k_I(e^{\eta_I}-1),\; 
\end{align}

we obtain
\begin{align}\label{e3.71}
  \partial_t\eta_I=P^*-P(t)-k_I\Gamma(t)(\frac{\gamma(t)}{K(t)}-\frac{1}{k_I})+k_I\Gamma^*(\frac{\gamma^*}{K^*}-\frac{1}{k_I})-\frac{\Gamma(t)\gamma(t)}{K(t)}\phi_I(\eta_I),
\end{align}

\begin{proof}[of Theorem \ref{th3.6}]
Using the Lyapunov candidate $V_{I}$, we get

\begin{align}\label{e3.73}
    \dot V_{I}
=\phi_{I}(\eta_{I}(t))\;\dot\eta_{I}(t).
\end{align}

From the equation \eqref{e3.71}, we substitute:

\begin{align}\label{e3.74}
    \dot V_{I}
= \phi_{I}(\eta_{I})\,\left[P^*-P(t)-k_I\Gamma(t)(\frac{\gamma(t)}{K(t)}-\frac{1}{k_I})+k_I\Gamma^*(\frac{\gamma^*}{K^*}-\frac{1}{k_I}) - \frac{\Gamma(t)\gamma(t)}{K(t)}\phi_{I}(\eta_{I})\right].
\end{align}
By choosing a control of the form \eqref{e3.75}
and thanks to {\bf Assumption (H1)}, it follows that
\begin{align}\label{e3.76}
    \dot V_{I}
= -\frac{\Gamma(t)\gamma(t)}{K(t)}\phi_{I}(\eta_{I})^{2}
\end{align}
holds; consequently, system \eqref{eq3.1}  is asymptotically stable.  
\end{proof}

\begin{remark}
The time derivative of the Lyapunov function $V_I$ can be written equivalently in quadratic form as

 \begin{align}\label{e3.77}
     \dot V_I(\eta)
 \;=\;
 -\,\bigl[\phi_{I}\;\;\phi_{I}\bigr]
  \;Q(t)\;
  \begin{bmatrix}\phi_{I} \\[4pt]\phi_{I}\end{bmatrix},\;\text{where}\;  Q(t) \;=
 \begin{pmatrix}
 \frac{\Gamma(t)^2\gamma(t)}{K(t)(\Gamma(t)+\gamma(t)-2K(t))} & -\; \frac{\Gamma(t)\gamma(t)}{\Gamma(t)+\gamma(t)-2K(t)}\\[4pt]
 -\; \frac{\Gamma(t)\gamma(t)}{\Gamma(t)+\gamma(t)-2K(t)}                & \frac{\Gamma(t)\gamma(t)^2}{K(t)(\Gamma(t)+\gamma(t)-2K(t))}
 \end{pmatrix}.
 \end{align}

 For $Q(t)$ to be positive definite, one requires

\begin{align}
\begin{cases}\label{equation7.9}
K(t)^2<\Gamma(t)\gamma(t),\quad\text{a.e.}\;t\in\mathbb{R}_+,\\
\\
2K(t)<\Gamma(t)+\gamma(t),\quad\text{a.e.}\;t\in\mathbb{R}_+.
\end{cases}
\end{align}
Consequently, its smallest eigenvalue is
\begin{align}\label{e3.79}
    \lambda_{\min}(Q(t))
 =\frac{2\Gamma(t)\gamma(t)}{K(t)}\frac{\Gamma(t)\gamma(t)-K(t)^2}{(\Gamma(t)+\gamma(t)-2K(t))(\Gamma(t)+\gamma(t))+\sqrt{(\Gamma(t)+\gamma(t)-2K(t))^2\left[(\Gamma(t)-\gamma(t))^2+4K(t)^2\right]}}.
\end{align}
Condition \eqref{equation7.9} implies the following relation
\begin{align}\label{e3.78}
\dfrac{K(t)}{\sqrt{\Gamma(t)}+1}<\sqrt{\gamma(t)+1}\quad t\in\mathbb{R}_+.
\end{align}
From \eqref{equation7.9}, there exists a constant $\delta_{\lambda}>0$ such that, for all $t$,
$\lambda_{\min}\big(Q(t)\big)\ge \delta_{\lambda}.$ Since the functions $\gamma(t),\;\Gamma(t)$ and $K(t)$ are bounded, we then obtain a constant $c>0$ such that, for all $t$, $\delta_{\lambda} \le \lambda_{\min}\big(Q(t)\big)\le c\;\text{a.e.}\; t\in\mathbb{R}_+.$
\end{remark}
\subsubsection{\bf Stability in the presence of a delay term (Theorem \ref{th3.8})}\label{s342}
We study the system’s stability when the kernel $\psi_I$, associated with the aquatic population $I$, is nonzero. This assumption takes into account the effect of the female population on the aquatic dynamics. Under {\bf Assumption (H5)}, we have from \eqref{e3.47}-\eqref{e3.55} 
\begin{align}\label{3.80}
   \partial_t\eta_I(t)=P^*-P(t)-k_I\Gamma(t)\left(\frac{\gamma(t)}{K(t)}-\frac{1}{k_I}\right)+k_I\Gamma^*\left(\frac{\gamma^*}{K^*}-\frac{1}{k_I}\right)-\frac{\Gamma(t)\gamma(t)k_I}{K(t)}\Biggl(\frac{e^{\eta_{I}}}{k_I}\int_{0}^{A}I^{*}(a)\bigl(1+\psi_{I}(t-a)\bigr)\,\mathrm{d}a-1\Biggr).
\end{align}
Define the normalized kernel

\begin{align}\label{e3.81}
g(a)
=\frac{I^{*}(a)}
{\displaystyle\int_{0}^{A}I^{*}(a)\,\mathrm{d}a},\; 
\quad
\int_{0}^{A}g(a)\,\mathrm{d}a=1,
\end{align}

Then
\begin{align}\label{e3.82}
    \frac{1}{k_I}\int_{0}^{A}I^{*}(a)\bigl(1+\psi_{I}(t-a)\bigr)\,\mathrm{d}a
=(1+\int_{0}^{A}g(a)\,\psi_{I}(t-a)\,\mathrm{d}a),
\end{align}
substituting gives

\begin{align}\label{e3.83}
    \partial_{t}\eta_{I}(t)
= P^*-P(t)-k_I\Gamma(t)(\frac{\gamma(t)}{K(t)}-\frac{1}{k_I})+k_I\Gamma^*(\frac{\gamma^*}{K^*}-\frac{1}{k_I})
\;-\;\frac{\Gamma(t)\gamma(t)}{K(t)}k_I\Bigl(e^{\eta_{I}}
\bigl[1+\!\int_{0}^{A}g(a)\psi_{I}(t-a)\,\mathrm{d}a\bigr]
-1\Bigr).
\end{align}
We therefore introduce the function
\begin{align}\label{e3.84}
\hat{\phi}_1= k\Bigl(
e^{\displaystyle\eta_{I}+\displaystyle\ln\bigl(1+\displaystyle\int_{0}^{A}g(a)\,\psi_{I}(t-a)\,\mathrm{d}a\bigr)}
-1\Bigr),
\end{align}
so that

\begin{align}\label{e3.85}
    \partial_{t}\eta_{I}
= P^*-P(t)-k_I\Gamma(t)(\frac{\gamma(t)}{K(t)}-\frac{1}{k_I})+k_I\Gamma^*(\frac{\gamma^*}{K^*}-\frac{1}{k_I}) \;-\frac{\Gamma(t)\gamma(t)}{K(t)}\;\hat\phi_{1}.
\end{align}

Finally, choosing the control \eqref{e3.75} yields
\begin{align}\label{e3.86}
    \partial_{t}\eta_{I}= \;-\frac{\Gamma(t)\gamma(t)}{K(t)}\;\hat\phi_{1},
\end{align}
with \(\psi_{i}\) defined in \eqref{e3.56}.

 \begin{proof}[of Theorem \ref{th3.8}]
Recall that

\begin{align*}
\dot V_{I}(\eta_I)
= -\frac{1}{2}\left(\bigl[\phi_{I}\;\;\phi_{I}\bigr]
  \;Q(t)\;
  \begin{bmatrix}\phi_{I} \\[4pt]\phi_{I}\end{bmatrix}+\bigl[\hat{\phi}_{1}\;\;\hat{\phi}_{1}\bigr]
  \;Q(t)\;
  \begin{bmatrix}\hat{\phi}_{1} \\[4pt]\hat{\phi}_{1}\end{bmatrix}\right)+\dfrac{\Gamma(t)\gamma(t)}{2K(t)}\Vert \hat{\phi}_{1}-\phi_I\Vert^2.
\end{align*}


Since $Q(t)$ is symmetric and positive semi-definite with strictly positive smallest eigenvalue $\lambda_{\min}(Q(t))>0$, it follows immediately that
\begin{align}\label{e93}
    \dot V_{I}(\eta_I)\;\le\; -\,\frac{\lambda_{\min}(Q(t))}{2}\,(\|\phi_I\|^{2}+\|\hat{\phi}_1\|^{2})+\dfrac{\Gamma(t)\gamma(t)}{2K(t)}\Vert \hat{\phi}_{1}-\phi_I\Vert^2.
\end{align}
and 
\begin{align}\label{e3.94}
    \Vert \hat{\phi}_{1}-\phi_I\Vert^2=(\phi_I+k_I)^2(e^{v_I}-1)^2,\;v_I=\displaystyle\ln\bigl(1+\displaystyle\int_{0}^{A}g(a)\,\psi_{I}(t-a)\,\mathrm{d}a\bigr)
\end{align}
then 
\begin{align}\label{e3.95}
\dot V_{1}(\eta_I)\;\le\; -\,\frac{\lambda_{\min}(Q(t))}{2}\,(\|\phi_I\|^{2}+\|\hat{\phi}_1\|^{2})+\dfrac{\Gamma(t)\gamma(t)}{2K(t)}(\phi_I+k_I)^2(e^{v_I}-1)^2.
\end{align}
For the second term $h(G_{I}(\psi_{I}))$, the Dini-derivative estimate \eqref{e3.90} implies
\begin{align}\label{e3.96}
    D^+h(G_I)=\frac{e^{G_I}-1}{G_I}\,D^+G_I\leq -\sigma_I (e^{G_I}-1).
\end{align}



By applying Young’s inequality  thanks to $\vert v_I\vert\leq G_I(\psi_I)$, we obtain




\begin{align}\label{e3.97}
    D^{+}V(\eta,\psi)
\le\;
 -\,\frac{3\lambda_{\min}(Q(t))}{4}\,\|\phi_I\|^{2}+\left[\frac{K(t)\lambda_{\min}(Q(t))+\Gamma(t)\gamma(t)}{2K(t)}(\phi_I+k_I)^2-\gamma_1\right](e^{G_I}-1)^2.
\end{align}
From {\bf Assumption (H1)}, we have
\begin{align}
\frac{K(t)\lambda_{\min}(Q(t))+\Gamma(t)\gamma(t)}{2K(t)}\leq \frac{\Vert K\Vert_{\infty}c+\Vert\Gamma\Vert_{\infty}\Vert\gamma\Vert_{\infty}}{2\epsilon}=:C,
\end{align}
then
\begin{align}
    D^{+}V(\eta,\psi)
\le\;
 -\,\frac{3\lambda_{\min}(Q(t))}{4}\,\|\phi_I\|^{2}+\left[C(\phi_I+k_I)^2-\gamma_1\right](e^{G_I}-1)^2.
\end{align}
Finally, with $(\eta_I,\psi_I)\in\mathcal{A}$, we obtain the required estimate, and hence the equilibrium is globally asymptotically stable. 
\end{proof}
\subsection{Proof of Theorem \ref{th5.4}}
\begin{lemma}
The nonlinear output \eqref{eq8.1} of \eqref{eq3.1}  is given by 
\begin{align}\label{eq8.2}
\begin{cases}
\dfrac{\partial_t y}{y}=\zeta_I-P(t)+\Gamma(t)-\frac{\Gamma(t)\gamma(t)}{K(t)}y(t)\displaystyle\int_{0}^A (1+\psi_I(t-a))p(a)da+\dfrac{\tilde{p}(0)\psi(t)-\tilde{p}(A)\psi(t-A)+\displaystyle\int_0^A\frac{d\tilde{p}(a)}{da}\psi(t-a)da}{q(\psi)}\\
y(0)=\displaystyle\int_0^Aw(a)I_0(a)da=y_0,\\
\psi_I(t)=\displaystyle\int_0^Ag_F(a)\psi_F(t-a)F^*(a)da,\;\psi_F(t)=\displaystyle\int_0^Ag_I(a)\psi_I(t-a)I^*(a)da,\\
 \psi_{i}(-a)
=\frac{i_0(a)}{i^{*}(a)\,\Pi[i_0]}\;-\;1
= \psi_{i,0}(a).
\end{cases}
\end{align}
with $p(a)=\dfrac{I^*(a)}{y^*},\;\tilde{p}(a)=w(a)p(a)$ and $q(\psi)=1+\displaystyle\int_0^A\tilde{p}(a)\psi(t-a)da,$ where the function $\tilde{p}$ satisfies $\tilde p(a)\geq 0$ for all $a\in (0,A)$ and $\displaystyle\int_0^A\tilde{p}(a)da=1.$
\end{lemma}
\begin{proof}
Using equation \eqref{e3.55} in \eqref{eq8.1} and then integrating with respect to age, the output state can be written as
\begin{align}
y(t)&=e^{\eta_I}\,y^*\,q(\psi).
\end{align}
Moreover, the following relation holds:
\begin{align}
\frac{\partial_t y}{y}&=\partial_t\eta_I+\partial_tq(\psi).
\end{align}
Finally, by replacing $I^*(a)e^{\eta_I}$ with $I_d(a,t)e^{\eta_I}$ in \eqref{e3.58}, one obtains the $\dfrac{\partial_t y}{y}$ term appearing in \eqref{eq8.2}
\end{proof}
In order to achieve the output $y$ matching the reference trajectory $y_d$, we replace the system output with $y_d$ in \eqref{eq8.2} and determine the corresponding control. Considering the convergence of $\psi_I$ to zero (see \ref{remark4.2}) , we then obtain
\begin{align}
P_{FF}(t)=\zeta_I-\dfrac{\partial_t y_d}{y_d}+\Gamma(t)-p^*\frac{\Gamma(t)\gamma(t)}{K(t)}y_d(t).
\end{align}
Let us the following lemma
\begin{lemma} Define $I_d(a,t)=p(a)y_d(t).$ Consider the following transformation
\begin{align}
 \left[\begin{array}{c}
v_I(t) \\ 
\psi_I(t-a)
\end{array}  \right]=\left[\begin{array}{c}
\ln[\tilde \Pi_I(I(t))] \\ 
\dfrac{I(a,t)}{I_d(a,t)\tilde\Pi_I(I(t))}-1
\end{array} \right],
\end{align}
where
\begin{align}
\tilde\Pi_I(I(t))=\dfrac{\langle \pi_{0,I}, I(t)\rangle_{L^2(0,A)}}{\langle\pi_{0,I}, p(a)y_d(t)\rangle_{L^2(0,A)}},
\end{align}

and $\pi_{0,I}$ the continuous function of \eqref{e3.54}. Moreover, the variables $\psi_i$ and $v_I$ satisfy:

\begin{align}\label{eq8.6}
\begin{cases}
    \partial_tv_I(t)=\zeta_I-P(t)-\dfrac{\partial_ty_d(t)}{y_d(t)}+\Gamma(t)-\frac{\Gamma(t)\gamma(t)}{K(t)}e^{v_I}\displaystyle\int_{0}^A (1+\psi_I(t-a))I_d(a,t)da,\\
v_{I}(0)
= \ln\!\bigl(\tilde\Pi[I_{0}]\bigr)
= v_{I0},
\end{cases}
\end{align}
and $\psi_I,\psi_F$ satisfying \eqref{e3.56}. The unique solution is given by 

\begin{align}\label{eq7.34}
I(a,t)
=I_d(a,t)e^{v_{I}(t)}\,\bigl(1+\psi_{I}(t-a)\bigr).
\end{align}
\end{lemma}
\begin{proof}
For the proof, we use the same strategy as in Lemma \ref{le3.4}.
\end{proof}
In the feedback control definition \eqref{equation4.12}, the expression $\ln\frac{y}{y_d}$ can be written in the form
\begin{align}
\displaystyle\ln\frac{y(t)}{y_d(t)}=v_I+\ln\left(1+\displaystyle\int_0^A\tilde{p}(a)\psi_I(t-a)da\right),
\end{align}
which serves as the basis for the controller design.
\begin{proof}[of Theorem \ref{th5.4}]
Decompose $P(t)=P_{FF}(t)+P_{FB}(t)$ and substitute $P_{FF}(t)$ into \eqref{eq8.6}. we obtain 
\begin{align}
\partial_t v_I(t)=\frac{\Gamma(t)\gamma(t)}{K(t)}p^*y_d(t)-\frac{\Gamma(t)\gamma(t)}{K(t)}y_d(t)\displaystyle\int_{0}^A p(a)e^{v_I}(1+\psi_I(t-a))da-sat_{\bar{\mathcal{U}}}(P_{FB}(t)),
\end{align}
\begin{align}\label{eq7.37}
\partial_tv_I(t)=\frac{\Gamma(t)\gamma(t)}{K(t)}\left[\displaystyle\int_0^AI_d(a,t)\left(1-e^{v_I}(1+\psi_I(t-a))\right)\right]-sat_{\bar{\mathcal{U}}}(P_{FB}(t)),
\end{align}
where $\bar{\mathcal{U}}=[\bar P_{\min},\bar P_{\max}]=[P_{\min}-P_{FF},P_{\max}-P_{FF}].$ From  \cite[Lemma 4.3]{ref69}, the inequality 
\begin{align}
\partial_t F(\psi_I) \leq -\sigma F(\psi_I)
\end{align}
holds for all $\psi_I \in \mathcal{S},$ where $F(\psi_I)=\dfrac{\max_{a\in(0,A)}(e^{-\sigma a}\vert\psi_I(t-a)\vert)}{1+\min(\min_{a\in(0,A)}(\psi_I(t-a)),0)}.$   Consider the functional
\begin{align}\label{eq7.39}
W(v_I,\psi_I)=v_I^2+\delta F(\psi_I),\;\delta>0\end{align}
and its time derivative of the form
\begin{align}\label{eq7.41}
\partial_tW=\partial_tv^2_I-\sigma\delta F(\psi_I).
\end{align}
Thanks to \eqref{eq7.37},  we have 
\begin{align}
\begin{aligned}
\partial_tv^2_I(t)=&2v_I\frac{\Gamma(t)\gamma(t)}{K(t)}\left[\displaystyle\int_0^AI_d(a,t)\left(1-e^{v_I}\right)\right]\\
&-2v_I\frac{\Gamma(t)\gamma(t)}{K(t)}\left[\displaystyle\int_0^AI_d(a,t)e^{v_I}\psi_I(t-a)\right]\\
&-2v_Isat_{\bar{\mathcal{U}}}(P_{FB}(t))
\end{aligned}
\end{align}
For $v_I<0,$ we have 
\begin{align}
\partial_tv_I^2\leq e^{\sigma A}y_d(t)p^*F(\psi_I)-2v_Isat_{\bar{\mathcal{U}}}(P_{FB}(t))
\end{align}
and by using the same strategy as in \cite{ref82}, we obtain
\begin{align}\label{eq7.43}
\partial_tv_I^2\leq \left(\dfrac{4v_I^2}{1+v_I}+2e^{\sigma A+1}F(\psi_I)\right)y_d(t)p^*-2v_Isat_{\bar{\mathcal{U}}}(P_{FB}(t))
\end{align}
for every $v_I.$ Also, the inequality
\begin{align}\label{eq7.44}
-2v_Isat_{\bar{\mathcal{U}}}(P_{FB}(t))\leq -\min(2,\alpha)\min(1,\bar P_{\min},\bar P_{\max})\dfrac{v_I^2}{1+\vert v_I\vert}+8\max(\bar P_{\min},\bar P_{\max})e^{\sigma A}F(\psi_I)
\end{align}
holds  from \cite[Lemma 6]{ref82}.  Upon substituting \eqref{eq7.43} into \eqref{eq7.41} and invoking \eqref{eq7.44}, we obtain 
\begin{align}\label{eq7.45}
\partial_tW\leq \mu_1\dfrac{v_I^2}{1+v_I}-\mu_2F(\psi_I)
\end{align}
where, for appropriate positive constants,
\begin{align}
\begin{aligned}
\mu_1=&\min(2,\alpha)\min(1,\bar P_{\min},\bar P_{\max})-4y_dp^*,\\
\mu_2=&\sigma\delta-8\max(\bar P_{\min},\bar P_{\max})e^{\sigma A}-2e^{\sigma A+1}y_d(t)p^*.
\end{aligned}
\end{align}
Choosing $\mu_\geq 0$ and $\mu_2>0,$ we get
\begin{align}
y_d(t)\in\left(0,\frac{\min(2,\alpha)}{4p^*}\min\left(1,\frac{P_{\max}-P_{\min}}{2}\right)\right]
\end{align} 
\begin{align}
\delta>&\frac{e^{\sigma A}}{\sigma}\left(8\max(\bar P_{\min},\bar P_{\max})+2ey_d(t)p^*\right)
\end{align}
and we define
\begin{align}
P_{FF}(t)\in \left[\frac{4p^*}{\min(2,\alpha)}y_d(t)+P_{\min},P_{\max}-\frac{4p^*}{\min(2,\alpha)}y_d(t)\right].
\end{align}
From \eqref{eq7.45}, we have 
\begin{align}
\begin{aligned}
\partial_tW\leq &\mu_1\dfrac{v_I^2}{1+\sqrt{\vert v_I\vert^2+\delta F(\psi_I)}}-\mu_2\dfrac{F(\psi_I)}{1+\sqrt{\vert v_I\vert^2+\delta F(\psi_I)}}
\end{aligned}
\end{align}
From \eqref{eq7.39} we have $F(\psi_I)=\frac{1}{\delta}(W-v_I^2),$ and 
\begin{align}
\mu_1v_I^2+\mu_2F(\psi_I)=\mu_1v_I^2+\frac{\mu_2}{\delta}(W-v_I^2)\geq \underbrace{\min(\mu_1,\frac{\mu_2}{\delta})}_{L}W,
\end{align}
thus, we obtain the structural inequality
\begin{align}\label{eq7.51}
\partial_tW\leq -L\dfrac{W}{1+\sqrt{W}}.
\end{align}
Integrating this inequality (see \cite{ref69}) yields 
\begin{align}\label{eq7.52}
\begin{aligned}
v_I^2\leq &\bar W(v_{I0},\psi_{I0})e^{\max(0,\bar W(v_{I0},\psi_{I0})-1)}e^{-\frac{L}{2}t},\\
\delta F(\psi_I)\leq &\bar W(v_{I0},\psi_{I0})e^{\max(0,\bar W(v_{I0},\psi_{I0})-1)}e^{-\frac{L}{2}t}.
\end{aligned}
\end{align}
Using \eqref{eq7.34} and the fact that $\vert\ln(1+x)\vert\leq \dfrac{\vert x\vert}{1+\min(0,x)}$ for all $x>-1,$ we get for every $a\in (0,A)$
\begin{align}
\displaystyle\left\vert\ln\left(\frac{I(a,t)}{I_d(a,t)}\right)\right\vert\leq\vert v_I\vert+\dfrac{\vert\psi_I(t-a)\vert}{1+\min(\psi_I(t-a),0)}\leq \vert v_I\vert+\dfrac{e^{\sigma A}\max_{a\in(0,A)}(e^{-\sigma a}\vert\psi_I(t-a)\vert)}{1+\min(\min_{a\in(0,A)}(\psi_I(t-a)),0)}
\end{align}
\begin{align}
\displaystyle\left\Vert\ln\left(\frac{I(a,t)}{I_d(a,t)}\right)\right\Vert_{\infty}\leq \sqrt{\vert v_I\vert^2}+e^{\sigma A} F(\psi_I),
\end{align}
and from \eqref{eq7.52}
\begin{align}
\displaystyle\left\Vert\ln\left(\frac{I(a,t)}{I_d(a,t)}\right)\right\Vert_{\infty}\leq \left(\sqrt{\bar W(v_{I0},\psi_{I0})}+\dfrac{e^{\sigma A}}{\delta}\bar W(v_{I0},\psi_{I0})\right)e^{\max(0,\bar W(v_{I0},\psi_{I0})-1)}e^{-\frac{L}{4}t}.
\end{align}
 Set  $\kappa_W(s)=(\sqrt{s}+\frac{e^{\sigma A}}{\delta}s)e^{\max(0,s-1)},$ then we obtain   
 \begin{align}
\displaystyle\left\Vert\ln\left(\frac{I(a,t)}{I_d(a,t)}\right)\right\Vert_{\infty}\leq \kappa_W(\bar W(v_{I0},\psi_{I0}))e^{-\frac{L}{4}t}
 \end{align}
 Express $\kappa_W$ as a function of the initial conditions. We have 
\begin{align}
v_{I0}\leq  \displaystyle\left\Vert\ln\left(\frac{I_0(a)}{I_d(a,0)}\right)\right\Vert_{\infty}=:\xi=:\kappa_{v_I}(\xi)
\end{align}
\begin{align}
F(\psi_I)\leq e^{2\xi}(e^{2\xi}-1)=:\kappa_{\psi}(\xi)
\end{align}
Thus, we obtain estimate \eqref{eq5.11} with $\kappa(\xi)=\kappa_W(\bar W(\kappa_{v_I}(\xi),\kappa_{\psi}(\xi))).$
\end{proof}
\begin{remark}
Since the original system is nonlinear, the stability analysis of the linearized system \eqref{eq3.11}, whose zero eigenvalue corresponds to the equilibrium profiles $(I^*(a),F^*(a), M^*(a))$, is not sufficient to guarantee overall stability. We therefore employ a nonlinear method based on the study of an adjoint mode. To this end, we introduce an adjoint eigenfunction ($\pi_{0,F},\pi_{0,I}, \pi_{0,M}$) associated with the zero eigenvalue of the linearized operator.  This eigenfunction acts as a filter : it allows us to project the nonlinear perturbations onto the critical age–time direction corresponding to the neutral spectral subspace. By projecting in $L^2(0,A)$, we exactly isolate the neutral mode of the dynamics \eqref{e3.63}, rather than applying an arbitrary projection. Projecting the full nonlinear dynamics onto this mode reduces the problem to a single ordinary differential equation (ODE) \eqref{e3.55} for the perturbation amplitude. Analyzing this ODE, then determines the asymptotic stability of the equilibrium. Hence, the adjoint eigenfunction is essential for completing the nonlinear analysis beyond what a mere linear spectral study can reveal. 
\end{remark}

\begin{remark}\label{re7.5}
By replacing the time-varying functions $K(t)$, $\Gamma(t)$, and $\gamma(t)$ with their average values $K^*$, $\Gamma^*$, and $\gamma^*$, in \eqref{e3.75} we recover identical global asymptotic stability results under the static control strategy 
\begin{align}\label{e3.99}
P(t)=P^*,\; \text{a.e.}\;t\in \mathbb{R}_+.
\end{align}
In the autonomous formulation, where all parameters are held constant, these fixed values provide a baseline for model analysis and equilibrium evaluation. 
\end{remark}
\subsection*{\bf Discussion on the control strategy \texorpdfstring{$P(t)$}{P(t)}}
In the control law \eqref{e3.75},
\begin{align}
P(t)=P^*+(\Gamma(t)-\Gamma^*)+k_I\Big(\frac{\Gamma^*\gamma^*}{K^*}-\frac{\Gamma(t)\gamma(t)}{K(t)}\Big),
\end{align}
the term $\Gamma(t)-\Gamma^*$ plays the role of a direct feedforward correction : it instantaneously compensates any variation in the growth rate $\Gamma$ and follows its phase shifts. The proportional term is explicitly
\begin{align}
\,k_I\Big(\dfrac{\Gamma^*\gamma^*}{K^*}-\dfrac{\Gamma(t)\gamma(t)}{K(t)}\Big)\,,
\end{align}
which constitutes a feedback mechanism on the normalized demographic pressure $\dfrac{\Gamma\gamma}{K}$. This term acts as a sensor of "demographic energy" and tends to drive the current pressure back toward its nominal value. The parameter $k_I$, representing the static aquatic total population, scales the feedback response : the larger $k_I$, the stronger the control reacts to deviations in demographic pressure. Hence, $k_I$ determines the intensity of the control effort applied. Under periodic forcing, $P(t)$ both tracks and attenuates parametric variations rather than allowing them to amplify the aquatic population; this results in a reduction of the oscillation amplitude of $I(a,t)$ and in a re-centering of the dynamics around a periodic equilibrium of smaller amplitude. By construction, $P(t)$ aims to restore the normalized pressure and promotes global asymptotic stability around the target state $I^*(a)$. The rigorous proof of this stabilization relies on the invariance of an appropriate attraction region (e.g., $\mathcal A$ defined in \eqref{e3.98}) and on the proper choice of $k_I$ and $P^*$, which guarantee the positivity and effectiveness of the control.

\begin{remark}
The condition \eqref{e3.78} defines a criterion ensuring the global stability of the equilibrium. When the growth rate $\Gamma(t)$ is sufficiently high, such as at the beginning of a favorable period when dormant eggs hatch en masse and larvae emerge from their latent state, the population density increases rapidly. This rapid growth, followed by a progressive decrease in the availability of environmental resources $K(t)$, intensifies intra-aquatic competition $\gamma(t)$. The quadratic regulation term then dominates, preventing indefinite population expansion. This inequality \eqref{e3.78} therefore provides a clear operational criterion : reducing the carrying capacity $K(t)$ amplifies the effect of intra-aquatic competition $\gamma(t)$, shifting the system toward a stable regime. Moreover, this condition reflects the effectiveness of control measures modeled by $-I(a,t)P(t)$ : any intervention that increases aquatic mortality (via the control $P(t)$) is equivalent to decreasing $K(t)$ and enhancing\ $\gamma(t)$, thereby facilitating satisfaction of the condition. From a biological perspective, low resource availability intensifies competition among individuals, reducing the survival of aquatic stages. Control measures that reduce $K(t)$ therefore effectively increase $\gamma(t)$, promoting equilibrium stability and ensuring sustained control effectiveness.
\end{remark}

\begin{remark}
The global stability analysis carried out in this section rigorously confirms that biological control targeting the aquatic stages of mosquito populations can, under specific structural conditions on the system’s parameters, lead to asymptotic stabilization of the equilibrium, thereby reflecting a sustained reduction in vector dynamics.

Beyond the theoretical results, several historical and contemporary examples support the effectiveness of such control strategies. A notable illustration is the case of Mandatory Palestine in the 1920s \cite{ref72}, where malaria was eliminated not through insecticides or vaccination, but primarily via the continuous destruction of larval breeding sites through systematic management of aquatic habitats, supported by community education and involvement.

Similar outcomes have been observed in regions such as Zanzibar, southern Tanzania, and rural India, where environmental sanitation, drainage, and the introduction of natural predators such as larvivorous fish have significantly reduced malaria transmission.

These observations, combined with our mathematical framework, suggest that biological control of mosquito aquatic populations is not only ecologically sustainable but also structurally effective in reducing malaria endemicity. This strategy, often complementary to chemical or genetic approaches, offers a relevant and efficient lever in malaria control policies, particularly in rural or semi-urban settings where continuous deployment of conventional interventions may be more challenging.
\end{remark}
\section{\bf Numerical simulation}\label{sec8}
 In this subsection we provide some numerical illustrations of our theoretical results. For the simulation of model \eqref{eq3.1}, the age discretization can be performed using the finite difference method on $(0,A).$ For more details, the reader is invited to consult \cite{ref58, ref57}. However, we perform here the numerical resolution by exploiting the strategies used in the theoretical analysis of stabilization. To do this, we first define the mortality functions $\mu_i$, the fertility functions $\beta$, the emergence rate $w$ as well as the functions $\gamma,\Gamma, K$. Next, we calculate $\zeta_i$ using the characteristic equation \eqref{e3.49}. Moreover, we calculate the new births $(I^*(0),F^*(0),M^*(0))$ using equation \eqref{e3.51}, which allows us to find the static densities $(I^*(a),F^*(a),M^*(a))$. Thus, the simulation result is based on the transformation \eqref{e3.58}. 
Then, for $\beta=3.68e^{-0.5a},\;w(a)=\frac{-a^2+4a}{16},\;\mu_I=\mu_F = \mu_M= 0.36e^{0.5a},\; A=4$, we obtain $\zeta_I = \zeta_F = 0.01.$ 
\begin{figure}[H]
    \centering
    \begin{subfigure}{0.32\textwidth}
        \centering
\includegraphics[width=\linewidth]{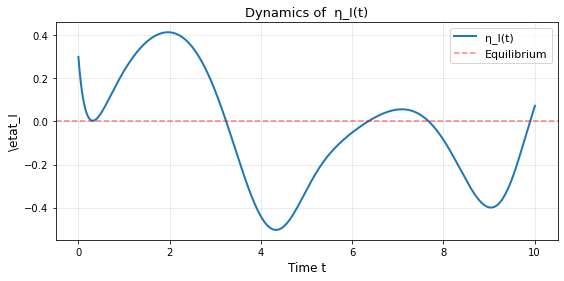}
    \end{subfigure}
    \hfill
    \begin{subfigure}{0.32\textwidth}
        \centering
\includegraphics[width=\linewidth]{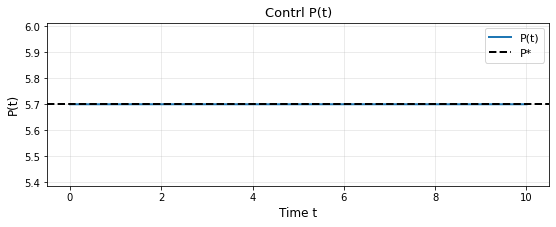}
    \end{subfigure}
    \hfill
    \begin{subfigure}{0.32\textwidth}
        \centering
\includegraphics[width=\linewidth]{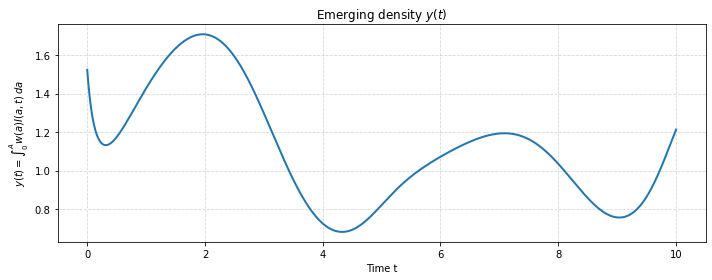}
    \end{subfigure}
    \hfill
    \begin{subfigure}{0.3\textwidth}
 \centering
        \includegraphics[width=\linewidth]{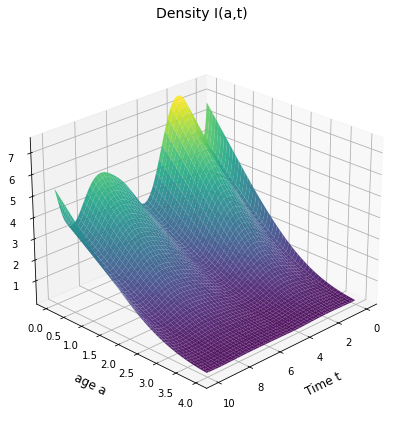}
    \end{subfigure}
     \hfill
    \begin{subfigure}{0.3\textwidth}
       \centering
        \includegraphics[width=\linewidth]{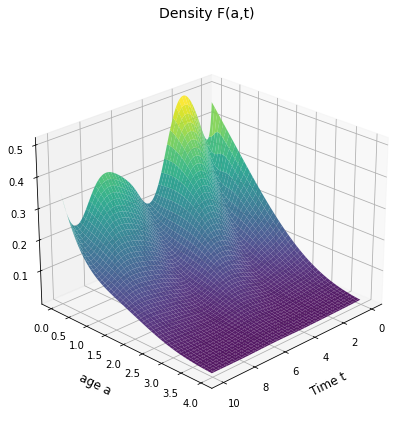}
    \end{subfigure}
     \hfill
    \begin{subfigure}{0.3\textwidth}
        \centering
        \includegraphics[width=\linewidth]{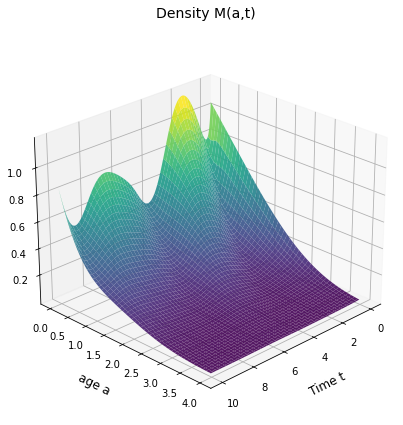}
    \end{subfigure}
\caption{With $\eta_{I0}=0.3$ as the initial condition, the environmental parameters are chosen to be periodic:
$K(t)=K^*(1+0.2\sin(3\pi t/T)),\; \Gamma(t)=\Gamma^*(1+0.3\sin(4\pi t/T))\; \text{and}\; \gamma(t)=\gamma^*(1+0.2\cos(3\pi t/T)).$
By choosing $P(t)=P^*$, we observe that the temporal heterogeneity of these covariates undermines the effectiveness of a static control $P^*$ in ensuring the system’s stability.}  
\end{figure}
\begin{figure}[H]
    \centering
    \begin{subfigure}{0.32\textwidth}
        \centering
\includegraphics[width=\linewidth]{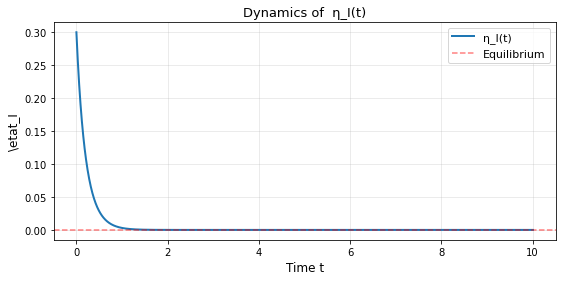}
    \end{subfigure}
    \hfill
    \begin{subfigure}{0.32\textwidth}
        \centering
\includegraphics[width=\linewidth]{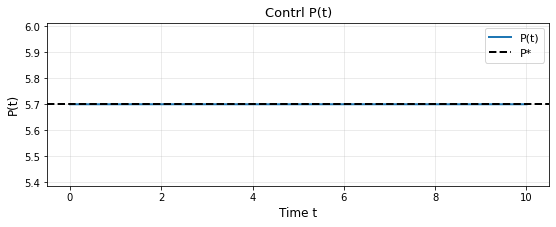}
    \end{subfigure}
    \hfill
    \begin{subfigure}{0.32\textwidth}
        \centering
\includegraphics[width=\linewidth]{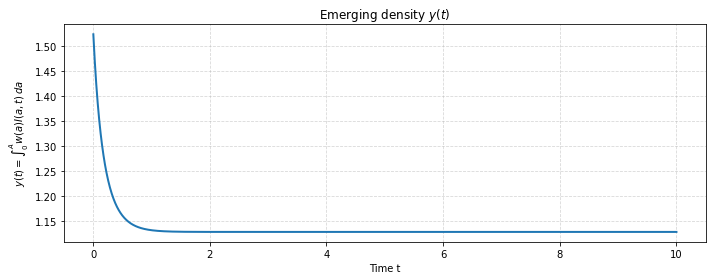}
    \end{subfigure}
    \hfill
    \begin{subfigure}{0.3\textwidth}
 \centering
        \includegraphics[width=\linewidth]{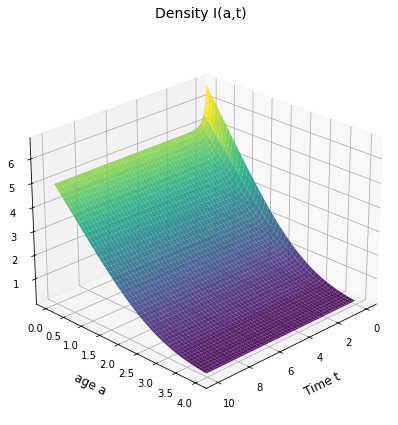}
    \end{subfigure}
     \hfill
    \begin{subfigure}{0.3\textwidth}
       \centering
        \includegraphics[width=\linewidth]{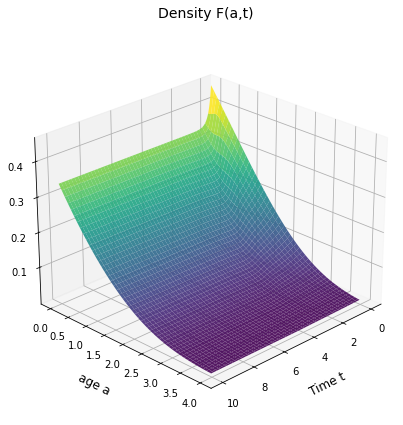}
    \end{subfigure}
     \hfill
    \begin{subfigure}{0.3\textwidth}
        \centering
        \includegraphics[width=\linewidth]{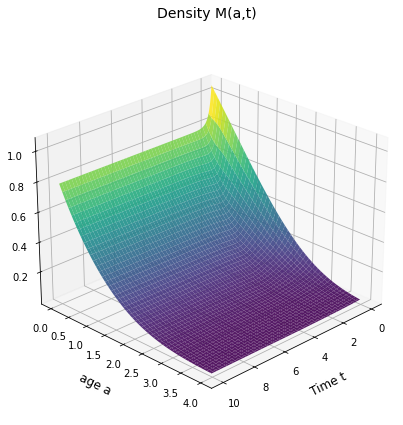}
    \end{subfigure}
\caption{Autonomous case : $\eta_{I0}=0.3\,\text{the initial condition}$ and $P(t)=P^*,\;K(t)=K^*,\;\Gamma(t)=\Gamma^*,\;\gamma(t)=\gamma^*$(see  Remark \ref{re7.5})} 
\end{figure}
\begin{figure}[H]
    \centering
    \begin{subfigure}{0.32\textwidth}
        \centering
\includegraphics[width=\linewidth]{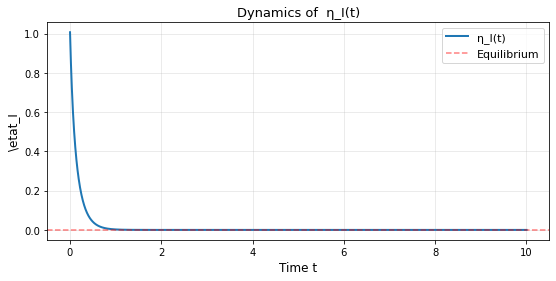}
    \end{subfigure}
    \hfill
    \begin{subfigure}{0.32\textwidth}
        \centering
\includegraphics[width=\linewidth]{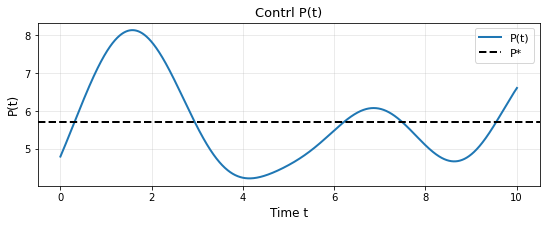}
    \end{subfigure}
    \hfill
    \begin{subfigure}{0.32\textwidth}
        \centering
\includegraphics[width=\linewidth]{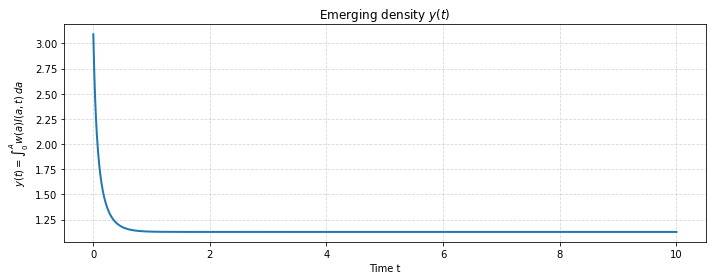}
    \end{subfigure}
    \hfill
    \begin{subfigure}{0.3\textwidth}
 \centering
        \includegraphics[width=\linewidth]{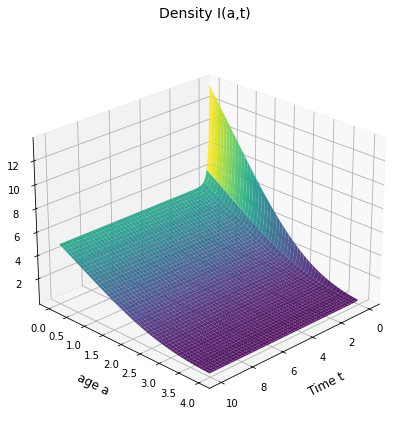}
    \end{subfigure}
     \hfill
    \begin{subfigure}{0.3\textwidth}
       \centering
        \includegraphics[width=\linewidth]{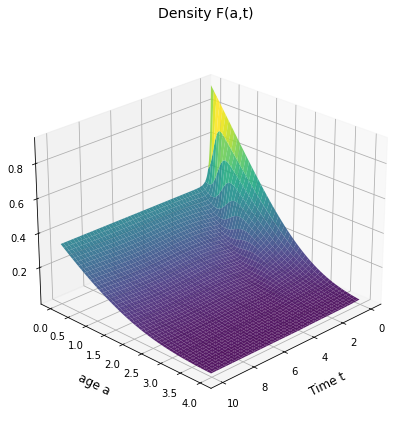}
    \end{subfigure}
     \hfill
    \begin{subfigure}{0.3\textwidth}
        \centering
        \includegraphics[width=\linewidth]{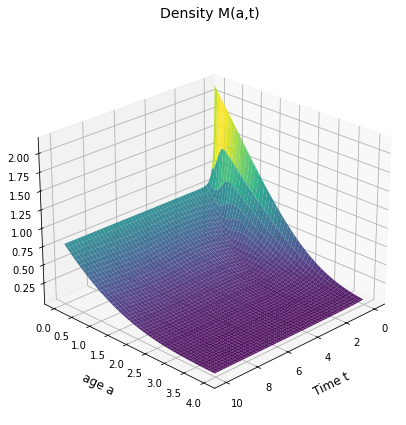}
    \end{subfigure}
\caption{The time-dependent (non-autonomous) case: the evolution of $\eta_I$ in \eqref{e3.55} under the time-dependent control given in \eqref{e3.75}, with initial value $\eta_{I0}=1.007.$} 
\end{figure}
\begin{figure}[H]
    \centering
    \begin{subfigure}{0.32\textwidth}
        \centering
\includegraphics[width=\linewidth]{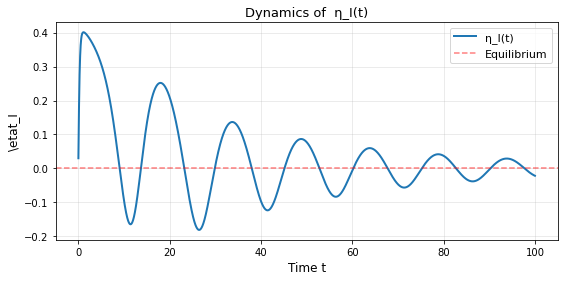}
    \end{subfigure}
    \hfill
    \begin{subfigure}{0.32\textwidth}
        \centering
\includegraphics[width=\linewidth]{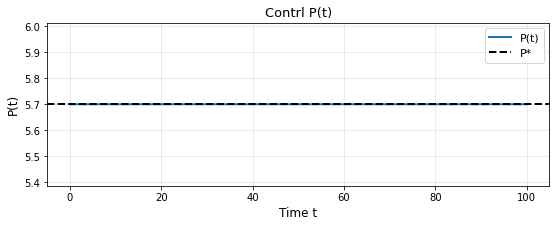}
    \end{subfigure}
    \hfill
    \begin{subfigure}{0.32\textwidth}
        \centering
\includegraphics[width=\linewidth]{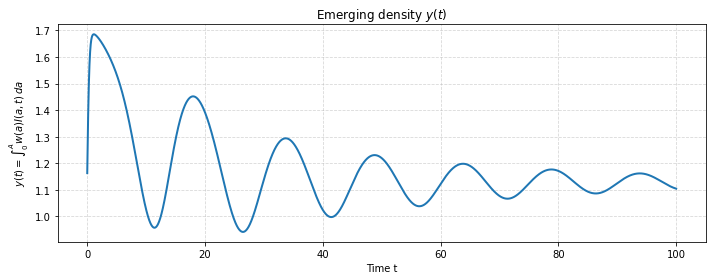}
    \end{subfigure}
    \hfill
    \begin{subfigure}{0.3\textwidth}
 \centering
        \includegraphics[width=\linewidth]{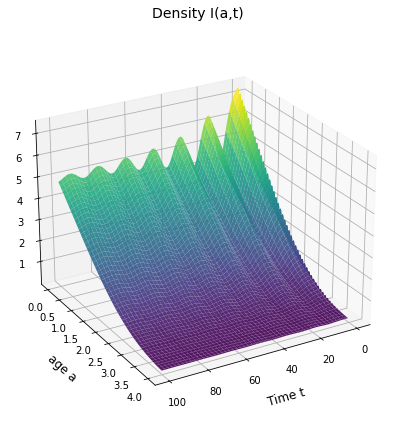}
    \end{subfigure}
     \hfill
    \begin{subfigure}{0.3\textwidth}
       \centering
        \includegraphics[width=\linewidth]{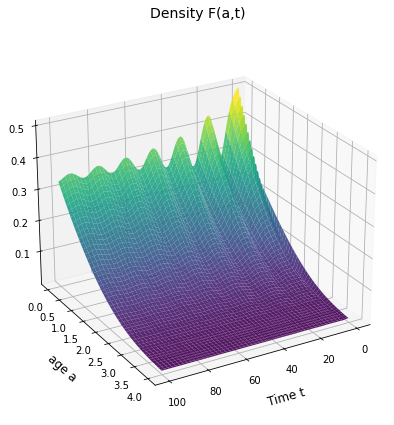}
    \end{subfigure}
     \hfill
    \begin{subfigure}{0.3\textwidth}
        \centering
        \includegraphics[width=\linewidth]{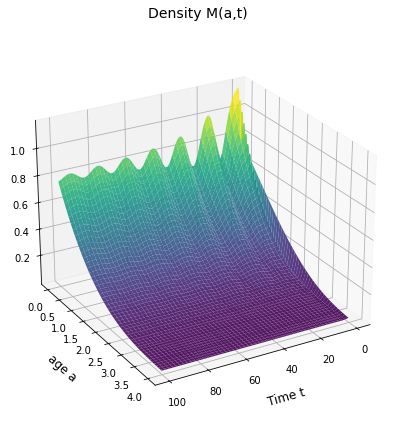}
    \end{subfigure}
\caption{$\eta_{I0}=0.03;\; K(t)=K^*+0.5K^*e^{-t/10},\; \Gamma(t)=\Gamma^ * (1.0 + 0.25 e^{-t / 40.0}\sin(2\pi t / 15.0)),\; \gamma(t)=\gamma^*(1+0.3e^{-t/8}\sin(2\pi t/20))$ with $P(t)=P^*.$}
\end{figure}
\begin{figure}[H]
    \centering
    \begin{subfigure}{0.32\textwidth}
        \centering
\includegraphics[width=\linewidth]{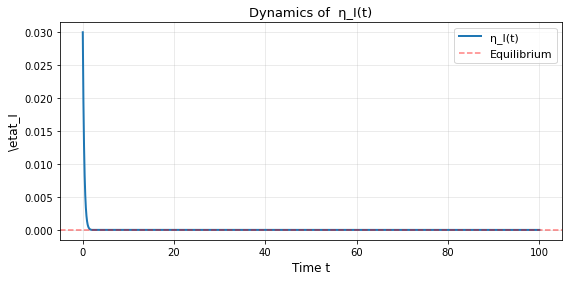}
    \end{subfigure}
    \hfill
    \begin{subfigure}{0.32\textwidth}
        \centering
\includegraphics[width=\linewidth]{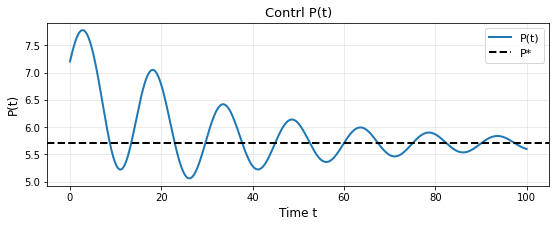}
    \end{subfigure}
    \hfill
    \begin{subfigure}{0.32\textwidth}
        \centering
\includegraphics[width=\linewidth]{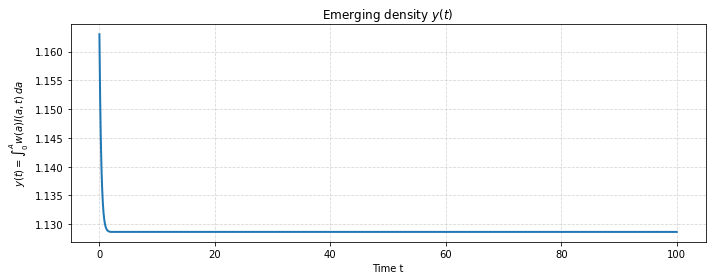}
    \end{subfigure}
    \hfill
    \begin{subfigure}{0.3\textwidth}
 \centering
        \includegraphics[width=\linewidth]{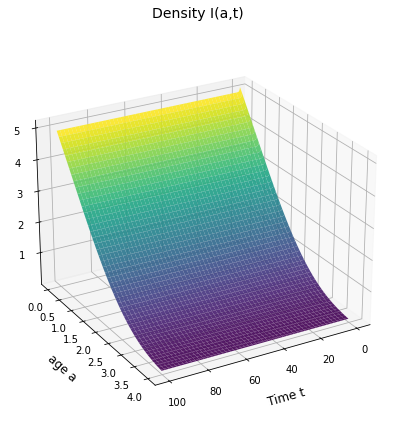}
    \end{subfigure}
     \hfill
    \begin{subfigure}{0.3\textwidth}
       \centering
        \includegraphics[width=\linewidth]{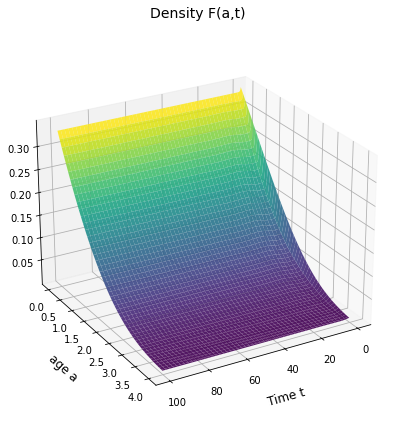}
    \end{subfigure}
     \hfill
    \begin{subfigure}{0.3\textwidth}
        \centering
        \includegraphics[width=\linewidth]{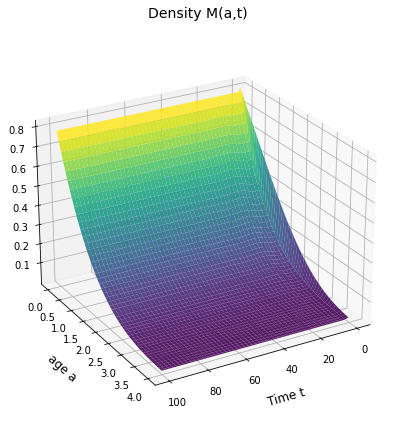}
    \end{subfigure}
\caption{$\eta_{I0}=0.03;\; K(t)=K^*+0.5K^*e^{-t/10},\; \Gamma(t)=\Gamma^ * (1.0 + 0.25 e^{-t / 40.0}\sin(2\pi t / 15.0)),\; \gamma(t)=\gamma^*(1+0.3e^{-t/8}\sin(2\pi t/20))$ with the control \eqref{e3.75}.}
\end{figure}
\begin{figure}[H]
    \centering
    \begin{subfigure}{0.45\textwidth}
        \centering
\includegraphics[width=\linewidth]{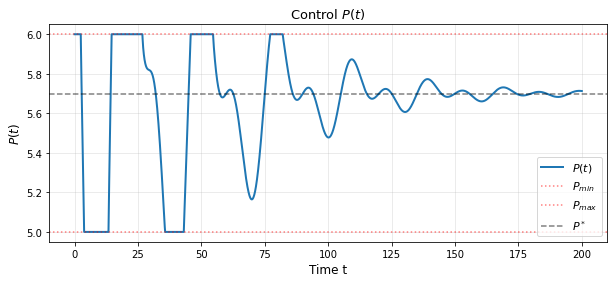}
        \caption{}
    \end{subfigure}
    \hfill
    \begin{subfigure}{0.45\textwidth}
        \centering
\includegraphics[width=\linewidth]{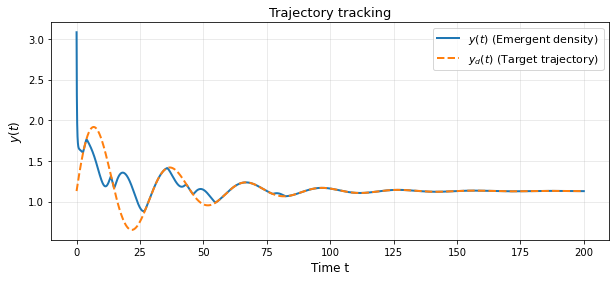}
       \caption{}
    \end{subfigure}
     \hfill
    \begin{subfigure}{0.45\textwidth}
        \centering
\includegraphics[width=\linewidth]{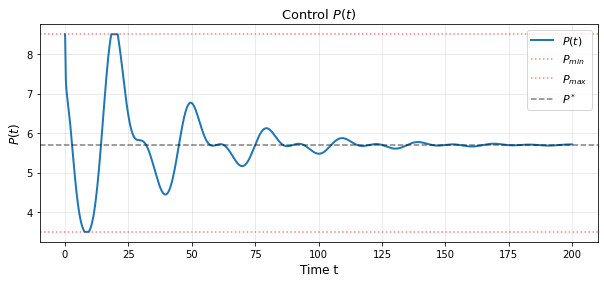}
       \caption{}
    \end{subfigure}
     \hfill
    \begin{subfigure}{0.45\textwidth}
        \centering
\includegraphics[width=\linewidth]{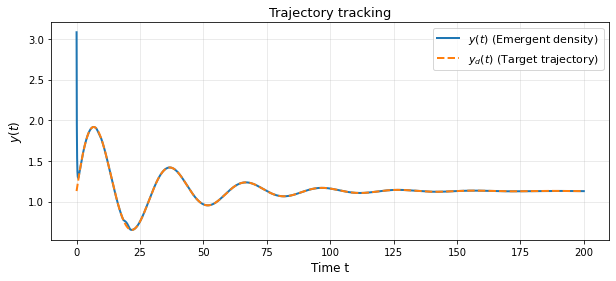}
        \caption{}
    \end{subfigure}
\caption{These figures illustrate the control of the emergent-stage density $y_d(t) = y^* + \sin\left(\frac{2\pi t}{30}\right) e^{-\frac{t}{30}}.$ The interval $(P_{\min},P_{\max})$ sets the amplitude of the available control action. The larger this range, the stronger and more rapid the actions the controller can apply; consequently $y$ converges more quickly to the target trajectory $y_d$ (see Figures C and D). Conversely, if the range is too narrow, the control action is limited and the system is frequently saturated (Figures A and B): the controller remains clamped at the bounds for a large fraction of the time, the response slows down and the convergence rate decreases. In practice, widening the range improves tracking speed, but this must be weighed against operational constraints.}
\end{figure}
\begin{figure}[H]
    \centering
    \begin{subfigure}{0.45\textwidth}
        \centering
\includegraphics[width=\linewidth]{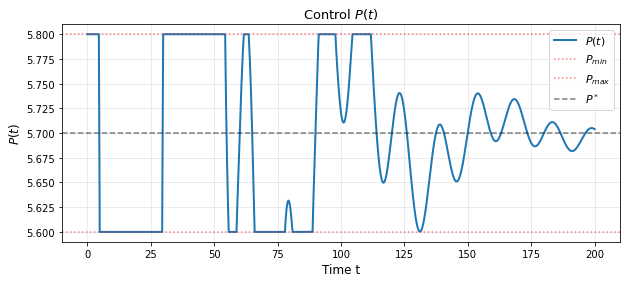}
        \caption{}
    \end{subfigure}
  \hfill
    \begin{subfigure}{0.45\textwidth}
        \centering
\includegraphics[width=\linewidth]{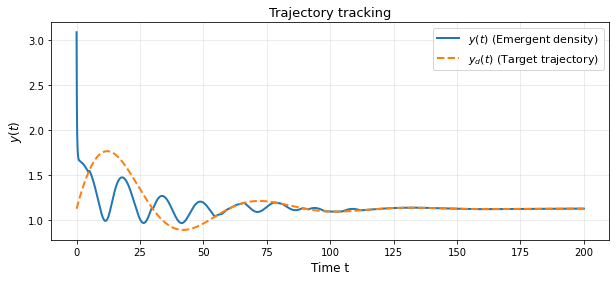}
       \caption{}
    \end{subfigure}
    \hfill
    \begin{subfigure}{0.45\textwidth}
        \centering
\includegraphics[width=\linewidth]{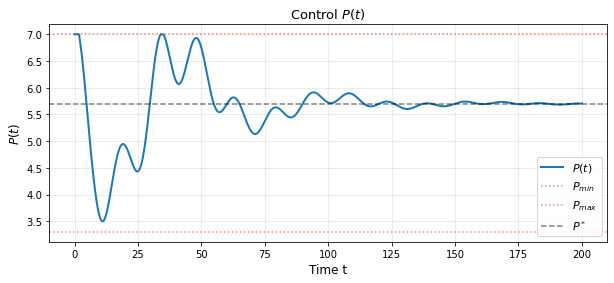}
       \caption{}
    \end{subfigure}
    \hfill
    \begin{subfigure}{0.45\textwidth}
        \centering
\includegraphics[width=\linewidth]{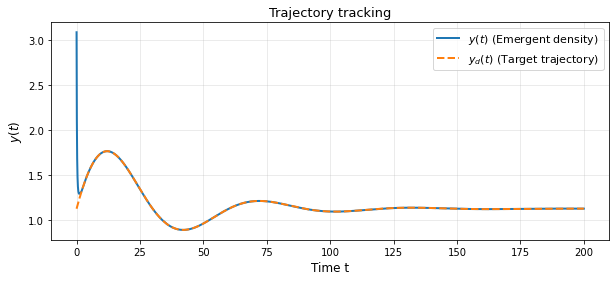}
       \caption{}
    \end{subfigure}
\caption{Figures A and B correspond to scenarios with a constrained control range ($P\in(5.6,5.8)$). Figures C and D correspond to scenarios with a relaxed control range. These comparisons underscore how the width of the admissible control set influences the speed and effectiveness of regulation of the emerging population.}
\end{figure}
\subsubsection*{\bf Comments and conclusions}
These numerical results highlight the influence and effectiveness (while maintaining a realistic character) of the control exerted on the aquatic population on the overall dynamics of system \eqref{eq3.1}. In the case of tracking the emerging population toward a desired trajectory $y(t)=y_d(t)$, the robustness of the control depends on the range of admissible control values $(P_{\min},P_{\max})$. The wider this interval, the greater the control’s margin for action and the more effectively it can ensure the convergence of $y$ toward $y_d$. In the presence of environmental covariates that make the model non-autonomous, a control law that responds proportionally to the temporal variations of the parameters (which significantly affect the dynamics) is sufficient to produce a marked influence on the overall dynamics.
\section{\bf Conclusion and outlook}\label{sec9}
In the context of the present work, the control of mosquito dynamics is based on regulating the emergence of the aquatic population density. The control of the emerging population is based on the design of a feedforward control and a feedback control associated with Lyapunov functions. The feedforward control ensures tracking of the desired state, while the feedback control guarantees the system’s stability in the face of various disturbances. This study shows that it is possible to reduce malaria transmission by lowering the density of mosquito vectors below a critical threshold through the control of the emergence of the aquatic population. It opens research avenues toward the study of controllability (exact or null) of the mosquito model \eqref{eq3.1} and, more broadly, toward the analysis of the turnpike property for mosquito dynamics. Finally, a natural extension is to incorporate the spatial dimension into the study of the dynamics and control.
\appendix
\section{\bf}\label{annexe:A}
\begin{proof}[of Lemma \ref{le3.4}]




By multiplying the equations of system \eqref{eq3.1} respectively by the functions $\pi_{0,I}$, $\pi_{0,F}$, and $\pi_{0,M}$, and then integrating by parts over the interval $(0, A)$, we obtain 
\begin{align}\label{e3.59}
     \left\langle \pi_{0,F}(a), \partial_ tF(a,t)\right\rangle= r\left\langle \pi_{0,F}(0)w(a),I(a,t)\right\rangle+\left\langle\partial_a\pi_{0,F}(a)-\pi_{0,F}(a)(\mu_F(a)+\zeta_F), F(a,t)\right\rangle +\left\langle\pi_{0,F}(a),(\zeta_F-\displaystyle\int_{0}^A \gamma(t)F(a,t)da)F(a,t)\right\rangle,
\end{align}

\begin{align}\label{e3.60}
     \left\langle \pi_{0,M}(a), \partial_ tM(a,t)\right\rangle= (1-r)\left\langle \pi_{0,M}(0)w(a),I(a,t)\right\rangle+\left\langle\partial_a\pi_{0,M}(a)-\pi_{0,M}(a)(\mu_M(a)+\zeta_M), M(a,t)\right\rangle
\end{align}
\begin{align*}
    +\left\langle\pi_{0,M}(a),(\zeta_M-\displaystyle\int_{0}^A \gamma (t)M(a,t)da)M(a,t)\right\rangle,
\end{align*}

\begin{align}\label{e3.61}
    \left\langle \pi_{0,I}(a),\partial_ tI(a,t)\right\rangle=\left\langle \pi_{0,I}(0)\beta(a,m),F(a,t)\right\rangle+\left\langle \partial_a\pi_{0,I}(a)-\pi_{0,I}(a)(\mu(a,p(t))+\zeta_I), I(a,t)\right\rangle+\left\langle\pi_{0,I}(a), (\zeta_I-P(t))I(a,t)\right\rangle
\end{align}

\begin{align*}
    +\left\langle\pi_{0,I}(a),(\Gamma(t)-\dfrac{\Gamma(t)\gamma(t)}{K(t)}\displaystyle\int_{0}^AI(a,t)da)I(a,t)\right\rangle.
\end{align*}
By summing them, we obtain

\begin{align}\label{e3.62}
    \left\langle \pi_{0,F}(a), \partial_ tF(a,t)-(\zeta_F-\displaystyle\int_{0}^A \gamma (t)F(a,t)da)F(a,t)\right\rangle+ \left\langle \pi_{0,I}(a),\partial_ tI(a,t)-(\zeta_I-P(t)+\Gamma(t)-\dfrac{\Gamma(t)\gamma(t)}{K(t)}\displaystyle\int_{0}^AI(a,t)da)I(a,t)\right\rangle+
\end{align}
\begin{align*}
    \left\langle \pi_{0,M}(a), \partial_ tM(a,t)-(\zeta_M-\displaystyle\int_{0}^A \gamma(t) M(a,t)da)M(a,t)\right\rangle=0,
\end{align*}
with 

\begin{align}\label{e3.63}
\begin{cases}
\mathcal{D}^*\pi_{0,I}(a)=\partial_a\pi_{0,I}(a)-\pi_{0,I}(a)(\mu(a,p(t))+\zeta_I)+r\pi_{0,F}(0)w(a)+(1-r)\pi_{0,M}(0)w(a),\quad \pi_{0,I}(A)=0,\\
\mathcal{D}^*\pi_{0,F}(a)=\partial_a\pi_{0,F}(a)-\pi_{0,F}(a)(\mu(a)+\zeta_F)+\pi_{0,I}(0) \beta(a,m),\qquad    \pi_{0,F}(A)=0,\\
\mathcal{D}^*\pi_{0,M}(a)=\partial_a\pi_{0,M}(a)-\pi_{0,M}(a)(\mu_M(a)+\zeta_M),\qquad \pi_{0,M}(A)=0.
\end{cases}
\end{align}
and $\pi_{0j}(a)=\displaystyle\int_a^{A}w(s)e^{\int_s^a(\zeta_j+\mu_j(l)dl}ds\; i\in \lbrace F\;M\rbrace.$ For all  functions $\pi_{0,F},\;\pi_{0,I},\; \pi_{0,M}$ in $L^2(0,A)$ implies that

\begin{align}\label{e3.64}
    \partial_t F(a,t)\;=\;\Bigl(\zeta_F-\!\!\int_{0}^{A}\gamma(t)\,F(a,t)\,da\Bigr)F(a,t),\;\quad
\partial_t M(a,t)\;=\;\Bigl(\zeta_M-\!\!\int_{0}^{A}\gamma(t)\,M(a,t)\,da\Bigr)M(a,t),
\end{align}
and 
\begin{align}\label{e3.65}
\partial_t I(a,t)\;=\;\Bigl(\zeta_I - P(t) + \Gamma(t) - \tfrac{\Gamma(t)\gamma(t)}{K(t)}\!\!\int_{0}^{A}I(a,t)\,da\Bigr)I(a,t), \quad\text{a.e.}\; (a,t)\in (0,A)\times (0,T).
\end{align}
Consequently
\begin{align*}
    \partial_t\eta_I(t)=\zeta_I-P(t)+\Gamma(t)-\dfrac{\Gamma(t)\gamma(t)}{K(t)}\displaystyle\int_{0}^AI(a,t)da=\zeta_I+R_I(t),
\end{align*}
and from transformation \eqref{e3.52}, we get 
\begin{align}\label{e3.66}
\partial_t\eta_I(t)=\zeta_I-P(t)+\Gamma(t)-\dfrac{\Gamma(t)\gamma(t)}{K(t)}e^{\eta_I}\displaystyle\int_{0}^A (1+\psi_I(t-a))I^*(a)da.
\end{align}
On the other hand, by definition
\begin{align}\label{e3.67}
    \psi_I(t)
 = \frac{I(0,t)\,e^{-\eta_I(t)}}{I^*(0)} \;-\; 1\Longrightarrow\psi_I(t)=\displaystyle\int_0^Ag_F(a)\psi_F(t-a)F^*(a)da.
\end{align}
By analogy, we obtain
\begin{align}\label{e3.68}
\begin{cases}
\psi_F(t)=\displaystyle\int_0^Ag_I(a)\psi_I(t-a)I^*(a)da,\\
\psi_M(t)= \displaystyle\int_0^Ag_I(a)\psi_I(t-a)I^*(a)da.
\end{cases}
\end{align}

By applying transformation \eqref{e3.52}, we obtain equation \eqref{e3.58}.


\end{proof}
\paragraph{\bf Acknowledgement}
\textbf{The authors wish to thank Prof. Enrique Zuazua  for his comments, suggestions and for fruitful discussions.}\\
\textbf{The author Yacouba Simpore is supported by the Alexander von Humboldt Foundation through an Alexander von Humboldt research fellowship.}
\bibliographystyle{plain}
 \bibliography{biblio_1}
 \vfill 
\end{document}